\documentclass[11pt]{amsart}

\usepackage{amsthm}
\usepackage{amsfonts}
\usepackage{amsmath}
\usepackage{amssymb}
\usepackage{enumitem}
\usepackage{fancyhdr}
\usepackage{tikz}
\usepackage{tikz-cd}
\usepackage{hyperref}

\usetikzlibrary{arrows,decorations.markings}

\usepackage[numbers]{natbib}

\usepackage{mathtools}

\setlist[description]{leftmargin=\parindent,labelindent=\parindent}

\usepackage{graphicx}
\graphicspath{ {images/} }

\newtheorem{theorem}{Theorem}
\newtheorem{lemma}[theorem]{Lemma}

\newtheorem{corollary}[theorem]{Corollary}

\numberwithin{theorem}{section}
\numberwithin{equation}{section}

\theoremstyle{definition}
\newtheorem{definition}[theorem]{Definition}

\newtheorem{remark}[theorem]{Remark}

\newtheorem{assumption}[theorem]{Assumption}

\title[$\infty$-inner products and OGW invariants]{Infinity inner products and open Gromov--Witten invariants}
\author{Sebastian Haney}
\address{Department of Mathematics, Columbia University, 2990 Broadway, New York, NY 10027}
\email{sebastian@math.columbia.edu}

\DeclareMathOperator\id{id}

\DeclareMathOperator\val{val}

\DeclareMathOperator\hol{hol}

\DeclareMathOperator\pt{pt}

\DeclareMathOperator\red{red}

\DeclareMathOperator\evi{evi}
\DeclareMathOperator\evb{evb}
\DeclareMathOperator\ev{ev}
\DeclareMathOperator\nov{nov}
\DeclareMathOperator\out{out}
\DeclareMathOperator\crit{crit}

\DeclareMathOperator\ind{ind}

\allowdisplaybreaks
\begin{document}
\begin{abstract}
The open Gromov--Witten (OGW) potential is a function from the set of weak bounding cochains on a closed Lagrangian in a closed symplectic manifold to the Novikov ring. Existing definitions of the OGW potential assume that the ground field of the Novikov ring is either $\mathbb{R}$ or $\mathbb{C}$. In this paper, we give an alternate definition of the OGW potential in the pearly model for Lagrangian Floer theory which yields an invariant valued in the Novikov ring over any field of characteristic zero. We work under simplifying regularity hypotheses which are satisfied, for instance, by any monotone Lagrangian. Our OGW potential is defined in terms of an appropriate weakening of a strictly cyclic pairing on a curved $A_{\infty}$-algebra, which can be thought of as a version of a proper Calabi--Yau structure. Such a structure is obtained by constructing a version of the cyclic open-closed map on the pearly Lagrangian Floer cochain complex. We also explain an analogue of our construction in de Rham cohomology, and show that it recovers the OGW potential constructed by Solomon and Tukachinsky.
\end{abstract}
\maketitle
\tableofcontents
\section{Introduction}
The genus zero Gromov--Witten invariants of a closed Lagrangian $L$ in a closed symplectic manifold $M$ should count pseudoholomorphic disks with boundary on $L$ in a way that is independent of the almost complex structure used to write the Cauchy--Riemann equation. In contrast to the theory of closed Gromov--Witten invariants, moduli spaces of pseudoholomorpic disks have boundary, which can potentially obstruct the invariance of open Gromov--Witten invariants. One expects the boundaries of these moduli spaces to contains pseudoholomorphic disks with boundary nodes and pseudoholomorphic spheres intersecting $L$. To account for disk bubbles, Joyce~\cite{Joy07} proposed that $L$ should be equipped with a bounding cochain as defined in~\cite{FOOO}.

For a graded Lagrangian $L$ in a Calabi--Yau threefold $M$, this idea was implemented by Fukaya in~\cite{F11}, where he constructed a generating function for the open Gromov--Witten invariants. Fukaya's open Gromov--Witten potential is a function from the space of bounding cochains on $L$ modulo gauge equivalence to the Novikov ring. The proof that Fukaya's open Gromov--Witten potential is gauge-invariant and invariant under changes of almost complex structure, up to a count of sphere bubbles, uses the existence of a cyclically symmetric pairing on the Fukaya $A_{\infty}$-algebra of $L$, i.e. an inner product $\langle\cdot,\cdot\rangle$ such that
\[ \langle\mathfrak{m}_k(\alpha_1,\ldots,\alpha_k),\alpha_{0}\rangle = \pm\langle\mathfrak{m}_k(\alpha_0,\ldots,\alpha_{k-1}),\alpha_k\rangle\]
up to a sign determined by the gradings of the inputs.

Geometrically, this symmetry should arise by cyclically permuting boundary constraints on pseudoholomorphic disks on $L$ and using the $S^1$-symmetry of the domain. In~\cite{F10}, Fukaya constructs an $A_{\infty}$-structure on the de Rham complex $\Omega^*(L)$, and shows that the integration pairing
\[ \langle\alpha,\beta\rangle = \pm\int_L\alpha\wedge\beta\]
(which is also defined up to a sign depending on the degrees of the inputs) is cyclically symmetric in this sense. Using this cyclic pairing, Fukaya defines the open Gromov--Witten potential of $L$ to be
\begin{align}
\Psi(b)\coloneqq\mathfrak{m}_{-1}+\sum_{k=0}^{\infty}\frac{1}{k+1}\langle\mathfrak{m}_k(a_1,\ldots,a_k),a_{0}\rangle
\end{align}
where $b$ is a bounding cochain on $L$ and $\mathfrak{m}_{-1}$ counts disks without boundary marked points.

The key technical input required in~\cite{F10} is the construction of a system of Kuranishi structures on the moduli spaces pseudoholomorphic disks with boundary on $L$, which are compatible with forgetful maps of marked points, and for which the evaluation map at one of the boundary marked points is a submersion. It is not known how to construct a Kuranishi structure which is compatible with forgetful maps and for which all evaluation maps at the boundary marked points are simultaneously submersions~\cite[Remark 3.2]{F10}. This lack of submersivity means that one cannot construct a cyclic pairing on the $A_{\infty}$-algebras of~\cite{FOOO}, which use smooth singular chains as a model for the cohomology of $L$. Consequently, it is unclear whether or not the Fukaya category generally carries a strictly cyclic structure over fields which do not contain $\mathbb{R}$. This implies that the open Gromov--Witten invariants of~\cite{F11} are only real-valued.

Solomon and Tukachinsky~\cite{ST21} explain how to extend Fukaya's construction of the open Gromov--Witten potential to Lagrangians of any dimension with possibly non-vanishing Maslov class by working over a ground ring with nontrivial grading. This construction recovers the open Gromov--Witten invariants defined by Welschinger~\cite{W13}~\cite{Che22b} and Georgieva~\cite{G16}. Their construction proceeds under the assumption that all relevant moduli spaces of disks are smooth orbifolds with corners, and that one of the boundary evaluation maps is a submersion. In the setting of~\cite{ST21}, one could conceivably impose additional submersivity assumptions to define cyclic $A_{\infty}$-algebra structures in characteristic zero, thus obtaining open Gromov--Witten invariants over any such field. Assuming submersivity of evaluation maps and regularity simultaneously, however, means that the invariants of~\cite{ST21} are only shown to be invariant under changes of almost complex structure in a weak sense. Specifically, their proof of invariance~\cite[Theorem 1]{ST21} requires that one has a path of \textit{regular} almost complex structures, which usually cannot be shown to exist by standard transversality arguments.

The purpose of this paper is to present a construction of the open Gromov--Witten potential over arbitrary fields of characteristic $0$ that does not require cyclic symmetry or any submersivity assumptions. To address the second of these problems, we define the Fukaya $A_{\infty}$-algebra using the Morse complex rather than the de Rham complex or smooth singular chains. Because the $A_{\infty}$-structures on the Morse complex counts configurations of pseudoholomorphic disks joined by Morse flow lines, it manifestly lacks cyclic symmetry.

Instead of using a cyclic pairing for our construction, we use the existence of a Calabi--Yau structure on the Fukaya category. One can define a Calabi--Yau structure on $\mathcal{A}$ to be an $A_{\infty}$-bimodule homomorphism from the diagonal bimdoule $\mathcal{A}_{\Delta}$ to the dual bimodule $\mathcal{A}^{\vee}$. Hence, a cyclic pairing can be thought of as a special case of a Calabi--Yau structure. It is a theorem of Kontsevich and Soibelman~\cite[Theorem 10.7]{KS} that, for an uncurved $A_{\infty}$-algebra, (weak proper) Calabi--Yau structures on $\mathcal{A}$ correspond to strictly cyclic pairings on a minimal model of $\mathcal{A}$.

In light of Kontsevich and Soibelman's theorem, one might hope to mimic the constructions of~\cite{F11} and~\cite{ST21} on a minimal model of the Fukaya $A_{\infty}$-algebra of $L$. The problem with this approach is that, in the filtered case, the potential of a cyclic $A_{\infty}$-algebras does not behave well under quasi-isomorphisms. In particular, it is only a quasi-isomorphism invariant up to additive constants~\cite{C08}. This phenomenon arises as one studies the dependence of the open Gromov--Witten potential on the choice almost complex structure used to define it, wherein such additive constants are given explicitly as counts of pseudoholomorphic teardrops.

Instead of passing to a cyclic minimal model, we construct the open Gromov--Witten potential by incorporating the higher order terms of the $A_{\infty}$-bimodule homomorphism to account for the lack of cyclic symmetry. More specifically, an $A_{\infty}$-bimodule homomorphism $\mathcal{A}_{\Delta}\to\mathcal{A}^{\vee}$ consists of a family of linear maps
\begin{align*}
\phi_{p,q}\colon\mathcal{A}^{\otimes p}\otimes\underline{\mathcal{A}}\otimes\mathcal{A}^{\otimes q}\to\mathcal{A}^{\vee}
\end{align*}
where the underline signifies that the corresponding factor of $\mathcal{A}$ is thought of as a bimodule over $\mathcal{A}$.

The potential $\Phi$ associated to a Lagrangian $L$ equipped with a bounding cochain $b$, defined over a field $\Bbbk$ of characteristic zero, is defined by the formula
\begin{align*}
\Phi(b)\coloneqq\mathfrak{m}_{-1}+\sum_{N=0}^{\infty}\sum_{p+q+k = N}\frac{1}{N+1}\phi_{p,q}(b^{\otimes p}\otimes\underline{\mathfrak{m}_k(b^{\otimes k})}\otimes b^{\otimes q})(b).
\end{align*}
Here, the structure coefficients $\lbrace\mathfrak{m}_k\rbrace_{k=0}^{\infty}$ arise from linear maps on the Morse complex $CM^*(L)$ defined over $\Bbbk$ and $\mathfrak{m}_{-1}$ is a count of rigid pearly trees with no inputs from the Morse cochain complex of $L$. The $A_{\infty}$-bimodule homomorphism appearing in this definition arises from a (possibly bulk-deformed) cyclic open-closed map on the Morse complex of $L$, in the sense of~\cite{Gan23}.

As we will show, this potential satisfies a wall-crossing formula of the same sort as the open Gromov--Witten potential in the cyclic case. 
\begin{theorem}[Paraphrasing of Theorem~\ref{truemain}]
The open Gromov--Witten potential is invariant under changes of almost complex structure up to a count of closed pseudoholomorphic spheres intersecting $L$ in a point.
\end{theorem}
Consequently, the open Gromov--Witten potential is defined over the field of definition of Maurer--Cartan elements of $L$.
\begin{corollary}\label{fieldofdefn}
If $L\subset M$ is a Lagrangian submanifold which is unobstructed by a bounding cochain $b$ defined over a field $\Bbbk$ of characteristic $0$, then its open Gromov--Witten potential is valued in $\Bbbk$.
\end{corollary}
Using the open Gromov--Witten potential, one can extract open Gromov--Witten invariants of certain Lagrangians whose spaces of bounding cochains are sufficiently well-understood. In the situations for which~\cite{ST21} define open Gromov--Witten invariants, one can in fact show that they are unobstruced over $\mathbb{Q}$. By Corollary~\ref{fieldofdefn}, we can guarantee the rationality of these invariants.
\begin{corollary}
If $L\subset M$ is a rational homology sphere (cf.~\cite[Theorem 2]{ST21}) or a real locus satisfying the conditions of~\cite[Theorem 3]{ST21}, then there exist $\mathbb{Q}$-valued open Gromov--Witten invariants for $L$.
\end{corollary}
Computations of relative period integrals carried out by Walcher~\cite{W12} lead to the prediction that there should exist graded Lagrangian submanifolds of the quintic threefold whose open Gromov--Witten invariants are irrational.  Rather than being evidence for the non-existence of a Calabi--Yau structures on the Fukaya category over $\mathbb{Q}$, Corollary~\ref{fieldofdefn} indicates that bounding cochains on these putative Lagrangians can only be constructed after passing to an appropriate field extension of $\mathbb{Q}$.

To keep our exposition simple and self-contained, we have only constructed the open Gromov--Witten potential under regularity assumptions which say that all moduli spaces of pearly trees needed for our construction are cut out transversely. These assumptions are achieved for any Lagrangian considered in~\cite{ST21} and for any monotone Lagrangian~\cite{BC07}, as explained in Appendix~\ref{regappendix}, where a complete list of all assumptions introduced throughout this paper can be found.

Because we have avoided using cyclic symmetry, our construction should also extend to cases where one defines the Fukaya $A_{\infty}$-algebra of $L$ using domain-dependent perturbations of the Cauchy--Riemann equation satisfying certain consistency conditions. In particular, this would allow for the definition of the open Gromov--Witten potential for Lagrangian immersions with clean self-intersections. This is a desirable generalization even if one is only interested in embedded Lagrangians, since considering Lagrangian immersions with disconnected domains enables easy constructions of nullhomologous Lagrangian immersions, for which the terms involving closed curves in the wall-crossing formula can be algebraically canceled by a choice of bounding chain.

\subsection*{Acknowledgments} I would like to thank Mohammed Abouzaid for helpful discussions about this project, and Jake Solomon for pointing out the relevance of~\cite{W12} to the existence of cyclic structures to me. I am also grateful to Andrei C\u{a}ld\u{a}raru for an interesting discussion about this work. This project was partially supported by the NSF through grant DMS-2103805.

\section{Hochschild invariants}\label{hochinvts}
In this section, we will collect some definitions pertaining to $A_{\infty}$-algebras, mainly for the purposes of fixing our notation and conventions, and review the definition of Hochschild and cyclic homology in the curved case. These invariants behave somewhat differently than for uncurved $A_{\infty}$-algebras, since the bar complex is no longer acyclic. Our primary use for this theory is to eventually extract a particular $A_{\infty}$-bimodule homomorphism $\mathcal{A}_{\Delta}\to\mathcal{A}^{\vee}$ from the diagonal bimodule over an $A_{\infty}$-algebra $\mathcal{A}$ to its dual over the ground field, so this does not pose a serious problem for us.
\subsection{\texorpdfstring{$A_{\infty}$}{A-infinity}-algebras}
All $A_{\infty}$-algebras we consider will be thought of as modules of a certain extension of the Novikov ring.
\begin{definition}
Let $\Bbbk$ be a field of characteristic $0$. For the formal variables $T$ of degree $0$ and $e$ of degree $2$, we denote by $\Lambda_{\nov}$ the universal Novikov ring
\begin{align}
\Lambda_{\nov}\coloneqq\left\lbrace\sum_{i=0}^{\infty}a_i T^{\lambda_i}e^{\mu_i}\colon a_i\in\Bbbk \, , \; \lambda_i\in\mathbb{R}_{\geq0} \, , \; \lim_{i\to\infty}\lambda_i = 0\right\rbrace.
\end{align}
\end{definition}
When we define Lagrangian Floer cohomology, it will actually be a module over a certain $\mathbb{Z}$-graded $\Lambda_{\nov}$-algebra, following~\cite{ST21}.
\begin{definition}
Let $s,t_0,\ldots,t_N$ be formal variables with integer gradings $|s|$ and $|t_i|$. Define the following $\mathbb{Z}$-graded and graded-commutative rings
\begin{align}
R &\coloneqq\Lambda_{\nov}[[s,t_0,\ldots,t_N]] \\
Q &\coloneqq\Lambda_{\nov}[[t_0,\ldots,t_N]].
\end{align}
The ring $R$ will be the coefficient ring for all $A_{\infty}$-algebras we consider. The ring $Q$ will be the ring of coefficients for bulk deformation classes, and thus does not appear explicitly in this section. There is a valuation on $R$ defined by
\begin{align*}
\nu\colon R &\to\mathbb{R} \\
\nu\left(\sum_{j=0}^{\infty}a_j T^{\lambda_j}e^{\mu_j}s^k\prod_{i=0}^N t_i^{\ell_{ij}}\right) &= \min_{\lbrace j\colon a_j\neq0\rbrace}\left(\lambda_j+k+\sum_{i=0}^N\ell_{ij}\right).
\end{align*}
\end{definition}
The extra variables in $R$ and $Q$ enable us to, for example, define the notion of a point-like bounding cochain.

Let $\mathcal{A}$ be a free graded $R$-module. It follows that there is a free graded $\Bbbk[[s,t_0,\ldots,t_N]]$-module $\overline{\mathcal{A}}$ such that
\[\mathcal{A} = \overline{\mathcal{A}}\otimes_{\Bbbk[[s,t_0,\ldots,t_N]]}R.\]
We denote by $|x|$ the grading of an element $x\in A$, and set $|x|' = |x|-1$. Note that the grading $|x|$ also incorporates the grading of coefficients in $R$.

In Lagrangian Floer theory, Gromov compactness implies that the $A_{\infty}$-algebras we will construct are gapped and filtered. To explain what this means, let $G\subset\mathbb{R}_{\geq0}\times 2\mathbb{Z}$ be a monoid such that
\begin{itemize}
\item[•] the image of $G$ in $\mathbb{R}_{\geq0}$ is discrete;
\item[•] $G\cap(\lbrace 0\rbrace\times 2\mathbb{Z}) = \lbrace(0,0)\rbrace\,$;
\item[•] for any $\lambda\in\mathbb{R}_{\geq0}$, the set $G\cap(\lbrace\lambda\rbrace\times 2\mathbb{Z})$ is finite.
\end{itemize}
For each $\beta = (\lambda(\beta),\mu(\beta))\in G$, suppose that we have a collection of linear maps
\begin{align}
\mathfrak{m}_{k,\beta}\colon(\overline{\mathcal{A}}[1])^{\otimes k}\to\overline{\mathcal{A}}[1]
\end{align}
for which $\mathfrak{m}_{0,(0,0)} = 0$. These induce $\Lambda_{\nov}$-linear maps
\begin{align}
&\mathfrak{m}_k\colon(\mathcal{A}[1])^{\otimes k}\to\mathcal{A}[1] \\
&\mathfrak{m}_k\coloneqq\sum T^{\lambda(\beta)}e^{\mu(\beta)/2}\mathfrak{m}_{k,\beta}.
\end{align}
The operations $\mathfrak{m}_k$ induce coderivations $\widehat{\mathfrak{m}}_k$ given by
\begin{align*}
\widehat{\mathfrak{m}}_k(x_1\otimes\cdots\otimes x_n) = \sum_{i=1}^{n-k}(-1)^{\maltese_i}x_n\otimes\cdots\otimes\mathfrak{m}_k(x_{i+1}\otimes\cdots\otimes x_{i+k})\otimes\cdots\otimes x_n
\end{align*}
for all $n\geq k$, where
\begin{equation}
\maltese_i\coloneqq\sum_{j=1}^i|x_j|' \label{cross}
\end{equation}
and setting $\widehat{\mathfrak{m}}_k(x_1\otimes\cdots\otimes x_n) = 0$ for $n<k$.
\begin{definition}
We say that $(\mathcal{A},\lbrace\mathfrak{m}_{k,\beta}\rbrace_{k = 0}^{\infty})$ form a \textit{gapped filtered $A_{\infty}$-algebra} if the coderivation
\[ \widehat{d} = \sum_{k=1}^{\infty}\widehat{\mathfrak{m}}_k \]
satisfies
\[ \widehat{d}\circ\widehat{d} = 0.\]
Equivalently, the operations $\mathfrak{m}_k$ are required to satisfy the curved $A_{\infty}$-relations
\begin{equation}
\sum_{i,\ell}(-1)^{\maltese_i}\mathfrak{m}_{k-\ell+1}(x_1\otimes\cdots\otimes\mathfrak{m}_{\ell}(x_{i+1}\otimes\cdots\otimes x_{i+\ell})\otimes\cdots\otimes x_k) = 0.\label{ainftyrels}
\end{equation}
\end{definition}
\begin{remark}[Sign conventions]
The sign~\eqref{cross} can be thought of as a Koszul sign arising when $\mathcal{A}$ acts on itself on the right. These are the sign conventions used b~\cite{C08},~\cite{C12},~\cite{CL11},~\cite{CL}, and~\cite{FOOO}. One can translate these signs to those of~\cite{Gan23} by replacing $\mathcal{A}$ with the opposite $A_{\infty}$-algebra.
\end{remark}
Additionally, $\mathcal{A}$ is said to be \textit{strictly unital} if there is an element $1\in\mathcal{A}$ with $|1| = 0$ such that
\begin{itemize}
\item $\mathfrak{m}_2(1,x) = x = (-1)^{|x|}\mathfrak{m}_2(x,1)$ and
\item $\mathfrak{m}_k(x_k,\ldots,1,\ldots,x_1) = 0$ whenever $k\neq 2$.
\end{itemize}
The Fukaya $A_{\infty}$-algebras we construct will possess strict units, but all of our arguments can be reworked in the homotopy unital setting~\cite[\S 3.3.1]{FOOO}. To compensate for the nonvanishing of $\mathfrak{m}_0$, we use weak bounding cochains, defined using the strict unit.
\begin{definition}
Let $\mathcal{A}$ be a strictly unital gapped filtered $A_{\infty}$-algebra and denote by $1$ its strict unit. For any $b\in\mathcal{A}$ with $|b| = 1$ and $\val(b)>0$, we say that it is a \textit{weak bounding cochain} if it satisfies the weak Maurer--Cartan equation
\[ \mathfrak{m}_0^b\coloneqq\sum_{k=0}^{\infty}\mathfrak{m}_k(b^{\otimes k}) = c\cdot 1. \]
Let $\widehat{\mathcal{M}}_{weak}(\mathcal{A})$ denote the set of bounding cochains on $\mathcal{A}$. If $b$ satisfies the Maurer--Cartan equation
\begin{equation}
\mathfrak{m}_0^b = 0
\end{equation}
it is said to be a \textit{bounding cochain}. Let $\widehat{\mathcal{M}}(\mathcal{A})$ denote the set of all bounding cochains on $\mathcal{A}$.
\end{definition}

We end this subsection by reviewing some basic notions related to filtered $A_{\infty}$-bimodules. Let $\mathcal{B}$ be a graded free filtered $R$-module, and fix a gapped filtered $A_{\infty}$-algebra $\mathcal{A}$ as above. An $A_{\infty}$-bimodule structure on $\mathcal{B}$ consists of a family of operations
\[ \mathfrak{n}_{p,q}\colon B_p(\mathcal{A}[1])\otimes\mathcal{B}[1]\otimes B_q(\mathcal{A}[1])\to\mathcal{B}[1]. \]
These maps induce
\[ \widehat{\delta}\colon\widehat{B}(\mathcal{A}[1])\widehat{\otimes}\mathcal{B}[1]\widehat{\otimes}\widehat{B}(\mathcal{A}[1])\to\widehat{B}(\mathcal{A}[1])\otimes\mathcal{B}[1]\widehat{\otimes}\widehat{B}(\mathcal{A}[1]) \]
defined by
\begin{align*}
&\widehat{\delta}(x_1\otimes\cdots\otimes x_k\otimes y\otimes z_1\otimes\cdots\otimes z_{\ell}) = \widehat{d}(x_1\otimes\cdots\otimes x_k)\otimes y\otimes z_1\otimes\cdots\otimes z_{\ell} \\
&+\sum(-1)^{\sum_{i=1}^{k-p}|x_i|'}x_1\otimes\cdots\otimes x_{k-p}\otimes\mathfrak{n}_{p,q}(x_{k-p+1}\otimes\cdots\otimes y\otimes\cdots\otimes z_q)\otimes\cdots\otimes z_{\ell} \\
\end{align*}
The maps $\lbrace\mathfrak{n}_{p,q}\rbrace_{p,q\geq0}$ give $\mathcal{B}$ the structure of an $A_{\infty}$-bimodule if
\[ \widehat{\delta}\circ\widehat{\delta} = 0. \]
The notion of an $A_{\infty}$-bimodule homomorphism (over the pair of $A_{\infty}$-algebra homomorphisms $(\id,\id)$ on $\mathcal{A}$) consists of a family of linear maps
\[ \phi_{p,q}\colon B_p(\mathcal{A}[1])\widehat{\otimes}\mathcal{B}[1]\widehat{\otimes} B_q(\mathcal{A}[1])\to\mathcal{B}'[1] \]
which respect the filtration on $\mathcal{B}$. We form 
\[\widehat{\phi}\colon B(\mathcal{A}[1])\widehat{\otimes}\mathcal{B}[1]\widehat{\otimes} B\mathcal{A}\to B(\mathcal{A}[1])\widehat{\otimes}\mathcal{B}'[1]\widehat{\otimes}B(\mathcal{A}[1])\]
by setting
\begin{align*}
&\widehat{\phi}(x_1\otimes\cdots\otimes x_k\otimes y\otimes z_1\otimes z_{\ell}) = \\
&\sum x_1\otimes\cdots\otimes x_{k-p}\otimes\phi_{p,q}(x_{k-p+1}\otimes\cdots\otimes y\otimes\cdots\otimes z_q)\otimes z_{q+1}\otimes\cdots\otimes z_{\ell}.
\end{align*}
The defining condition for an $A_{\infty}$-bimodule homomorphism is
\begin{equation}
\widehat{\phi}\circ\widehat{\delta} = \widehat{\delta}'\circ\widehat{\phi}
\end{equation}
where $\delta'$ is induced from the $A_{\infty}$-bimodule structure maps on $\mathcal{B}'$.

The two main examples of $A_{\infty}$-bimodule we will use are the diagonal bimodule $\mathcal{A}_{\Delta}$, and the dual bimodule $\mathcal{A}^{\vee}$. To define the latter, let $\mathcal{A}^{\vee}$ denote the $R$-dual of $\mathcal{A}$, and equip it with structure maps $\lbrace\mathfrak{m}^{\vee}_{k,\ell}\rbrace_{k,\ell\geq0}$ given by
\begin{align*}
&\mathfrak{m}^{\vee}_{k,\ell}(x_1\otimes\cdots\otimes x_k\otimes v^{\vee}\otimes x_{k+1}\otimes\cdots\otimes x_{\ell})(w) = \\
&(-1)^{\epsilon}v^{\vee}(\mathfrak{m}_{k+\ell+1}(x_{k+1}\otimes\cdots\otimes x_{k+\ell}\otimes w\otimes x_1\otimes\cdots\otimes x_k))
\end{align*}
where the sign is determined by
\[ \epsilon = |v^*|'+\left(\sum_{i=1}^k |x_i|'\right)\left(|v^*|'+\sum_{i=k+1}^{k+\ell}|x_i|'+|w|'\right). \]
\subsection{The Hochschild complex} Let $\mathcal{A}$ be a strictly unital gapped filtered $A_{\infty}$-algebra, and consider an $A_{\infty}$-bimodule $\mathcal{B}$ over $\mathcal{A}$. Denote by
\begin{align*}
CH_*^k(\mathcal{A},\mathcal{B})\coloneqq\underline{\mathcal{B}[1]}\otimes(\mathcal{A}[1])^{\otimes k}
\end{align*}
the space of Hochschild chains of length $k$, with degree given by
\[ |\underline{b}\otimes a_1\otimes\cdots\otimes a_k| = |b|+\sum_{i=1}^k |a_i|'.\]
We have underlined the bimodule factor for readability. Since $\mathcal{A}$ can be curved, it is usually more convenient to consider reduced Hochschild chains, which are given by
\begin{align*}
CH_*^{\red, k}(\mathcal{A},\mathcal{B})\coloneqq\underline{\mathcal{B}[1]}\otimes(\mathcal{A}[1]/R\cdot 1)^{\otimes k}.
\end{align*}
The Hochschild chain complex is the completed direct sum
\begin{align*}
CH_*(\mathcal{A},\mathcal{B})\coloneqq\widehat{\bigoplus}_{k\geq0} CH_*^k(\mathcal{A},\mathcal{B}).
\end{align*}
Similarly, the reduced Hochschild chain space is the completed direct sum
\begin{align*}
CH_*^{\red}(\mathcal{A},\mathcal{B})\coloneqq\widehat{\bigoplus}_{k\geq0} CH_*^{\red,k}(\mathcal{A},\mathcal{B}).
\end{align*}
For $v\in\mathcal{B}$ and $a_i\in\mathcal{A}$, the Hochschild differential is defined by
\begin{align*}
&b(\underline{v}\otimes a_{1}\otimes\cdots\otimes a_k)\coloneqq\\
&\sum(-1)^{\#_{j}^i}\underline{\mathfrak{m}_{i+j+1}(a_{k-i+1}\otimes\cdots\otimes a_k\otimes\underline{v}\otimes a_{1}\otimes\cdots\otimes a_{j})}\otimes a_{j+1}\otimes\cdots\otimes a_{k-i} \\
&+\sum(-1)^{\maltese_i'}\underline{v}\otimes a_1\otimes\cdots\otimes a_{i-1}\otimes\mathfrak{m}_j(a_i\otimes\cdots\otimes a_{i+j-1})\otimes\cdots\otimes a_k.
\end{align*}
The signs above are determined by
\begin{align*}
\#_j^i &\coloneqq\left(\sum_{s=1}^i|a_{k-i+s}|'\right)\left(|v|'+\sum_{t=1}^j|a_t|'\right) \\
\maltese_i' &\coloneqq|v|'+\sum_{s=1}^{i-1}|a_s|'.
\end{align*}
Notice that $\mathfrak{m}_0$ can appear in terms of the second sum in the definition of the Hochschild differential, but we still have that $b\circ b = 0$. The main examples we will consider are the cases where $\mathcal{B} = \mathcal{A}_{\Delta}$ is the diagonal bimodule, or where $\mathcal{B} = \mathcal{A}^{\vee}$ is the $\Lambda_{\nov}$-dual bimodule.

We use the strict unit on $\mathcal{A}$ to construct a homological $S^1$-action on
\[CH_*(\mathcal{A})\coloneqq CH_*(\mathcal{A},\mathcal{A}_{\Delta})\]
in the following sense. Let $t\in\mathbb{Z}/k\mathbb{Z}$ denote the generator, and define its action on $\mathcal{A}^{\otimes k}$ by
\begin{align*}
t(a_1\otimes\cdots\otimes a_k) = (-1)^{\maltese_{k-1}\cdot|a_k|'}a_k\otimes a_1\otimes\cdots\otimes a_{k-1}.
\end{align*}
Define an operator $N = 1+t+\cdots+t^{k-1}$ on
\[CH_*^k(\mathcal{A},\mathcal{A}_{\Delta}).\]
We abuse notation and write $t$ for the generator of any cyclic group, and $N$ for the operator $CH_*(\mathcal{A})$ obtained by considering the action of $N$ as above on each direct summand. Let $b' = \widehat{d}$ denote the bar differential. The maps defined so far satisfy the relations
\begin{align*}
b(1-t) &= (1-t)b' \\
b'N &= Nb
\end{align*}
meaning that we can form the following bicomplex.
\[\begin{tikzcd}
{} & {} & {} & {} & {} \\
{} & CH_1(\mathcal{A})\arrow{u}\arrow{l} & CH_1(\mathcal{A})\arrow{u}\arrow{l}{1-t} & CH_1(\mathcal{A})\arrow{u}\arrow{l}{N} & CH_1(\mathcal{A})\arrow{u}\arrow{l}{1-t} & {}\arrow{l} \\
{} & CH_0(\mathcal{A})\arrow{u}{b}\arrow{l} & CH_0(\mathcal{A})\arrow{u}{-b'}\arrow{l}{1-t} & CH_0(\mathcal{A}) \arrow{u}{b}\arrow{l}{N} & CH_0(\mathcal{A})\arrow{u}{-b'}\arrow{l}{1-t} & {}\arrow{l} \\
{} & CH_{-1}(\mathcal{A})\arrow{u}{b}\arrow{l} & CH_{-1}(\mathcal{A})\arrow{u}{-b'}\arrow{l}{1-t} & CH_{-1}(\mathcal{A})\arrow{u}{b}\arrow{l}{N} & CH_{-1}(\mathcal{A})\arrow{u}{-b'}\arrow{l}{1-t} & {}\arrow{l} \\
{} & {}\arrow{u} & {}\arrow{u} & {}\arrow{u} & {}\arrow{u}
\end{tikzcd}\]
All of these maps descend to $CH_*^{\red}(\mathcal{A})$, so we can define a reduced bicomplex analogously. To define an analogue of Connes' operator when $\mathfrak{m}_0$ is nonzero, we need to specify a slightly different contracting homotopy for the bar complex than is typically used in the uncurved case.
\begin{lemma}~\cite[Lemma 5.4]{C12}
There is a contracting homotopy $\tilde{s}$ for the complex 
\[(CH_*(\mathcal{A}),\widehat{d})\]
which is defined by decomposing
\begin{align*}
CH_*(\mathcal{A}) = \ker(\widehat{d})\oplus V
\end{align*}
for some $\Lambda_{\nov}$-module $V$, and setting
\begin{align*}
\widetilde{s}(\alpha) =
\begin{cases}
1\otimes\alpha \, , & \alpha\in V \\
1\otimes\alpha - \mathfrak{m}_0\otimes 1\otimes\alpha \, , & \alpha\in\ker\widehat{d}.
\end{cases}
\end{align*}
\qed
\end{lemma}
The Connes $B$ operator is then
\begin{align*}
B = (1-t)\widetilde{s}N
\end{align*}
and we can form the $(b,B)$-bicomplex.
\[\begin{tikzcd}
{} & {} & {} & {} & {}\\
{} & CH_1(\mathcal{A})\arrow{u}\arrow{l} & CH_2(\mathcal{A})\arrow{u}\arrow{l}{B} & CH_3(\mathcal{A})\arrow{u}\arrow{l}{B} & {}\arrow{l} \\
{} & CH_0(\mathcal{A})\arrow{u}{b}\arrow{l} & CH_1(\mathcal{A})\arrow{u}{b}\arrow{l}{B} & CH_2(\mathcal{A})\arrow{u}{b}\arrow{l}{B} & {}\arrow{l}\\
{} & CH_{-1}(\mathcal{A})\arrow{u}{b}\arrow{l} & CH_0(\mathcal{A})\arrow{u}{b}\arrow{l}{B} & CH_1(\mathcal{A})\arrow{u}{b}\arrow{l}{B} & {}\arrow{l}\\
{} & {}\arrow{u} & {}\arrow{u} & {}\arrow{u} & {}
\end{tikzcd}\]
On $CH_*^{\red}$, the terms of $\widetilde{s}$ of the form $\mathfrak{m}_0\otimes 1\otimes\alpha$ vanish, and thus $B$ descends to the usual Connes operator, i.e.
\[ B(a_1\otimes\cdots\otimes a_k) = \sum 1\otimes a_i\otimes\cdots\otimes a_k\otimes a_1\cdots\otimes a_{i-1} \]
on the reduced Hochschild chain space.
\begin{remark}
To account for possibly non-unital $A_{\infty}$-algebras, one could instead work with the non-unital Hochschild complex, which also carries a homological $S^1$-action~\cite[\S 3.2]{Gan23}. Since the Floer cochain complexes we construct are strictly unital, this will not be necessary in this paper.
\end{remark}
If $u$ is a formal variable of degree $2$, we define $b_{eq} = b+uB$, and define the positive, negative, and periodic cyclic chain complexes
\begin{align*}
CC_*^+(\mathcal{A}) &\coloneqq (CH_*(\mathcal{A})\otimes_{R}R((u))/uR[[u]],b_{eq}) \\
CC_*^{-}(\mathcal{A}) &\coloneqq (CH_*(\mathcal{A})\widehat{\otimes}_{R}R[[u]],b_{eq}) \\
CC_*^{\infty}(\mathcal{A}) &\coloneqq (CH_*(\mathcal{A})\widehat{\otimes}_{R}R((u)),b_{eq})
\end{align*}
respectively. The positive and negative cyclic chain complexes are obtained by restricting to the positive and negative columns of the $(b,B)$-complex, respectively, and the periodic cyclic complex is obtained from the full bicomplex. Denote by $HC_*^{+/-/\infty}(\mathcal{A})$ the homology of these complexes, called the positive, negative, or periodic cyclic homology. Similarly, one defines the reduced cyclic complexes $CC_{*}^{\circ,\red}(\mathcal{A})$ and the reduced cyclic cohomologies $HC_{*}^{\circ,\red}(\mathcal{A})$, where $\circ\in\lbrace +,-,\infty\rbrace$.

Dualizing the reduced $(b,B)$-complex, gives us the reduced $(b^*,B^*)$-complex
\begin{equation}\label{b*B*complex}
\begin{tikzcd}
{} & {} & {} & {} & {}\\
{}\arrow{r} & CH^2_{\red}(\mathcal{A},\mathcal{A}^{\vee})\arrow{r}{B^*}\arrow{u} & CH^1_{\red}(\mathcal{A},\mathcal{A}^{\vee})\arrow{r}{B^*}\arrow{u} & CH^0_{\red}(\mathcal{A},\mathcal{A}^{\vee}) \arrow{r}\arrow{u}& {} \\
{}\arrow{r} & CH^1_{\red}(\mathcal{A},\mathcal{A}^{\vee})\arrow{r}{B^*}\arrow{u}{b^*} & CH^0_{\red}(\mathcal{A},\mathcal{A}^{\vee})\arrow{r}{B^*}\arrow{u}{b^*} & CH^{-1}_{\red}(\mathcal{A},\mathcal{A}^{\vee})\arrow{r}\arrow{u}{b^*} & {} \\
{}\arrow{r} & CH^0_{\red}(\mathcal{A},\mathcal{A}^{\vee})\arrow{r}{B^*}\arrow{u}{b^*} & CH^{-1}_{\red}(\mathcal{A},\mathcal{A}^{\vee})\arrow{r}{B^*}\arrow{u}{b^*} & CH^{-2}_{\red}(\mathcal{A},\mathcal{A}^{\vee})\arrow{r}\arrow{u}{b^*} & {} \\
{} & {}\arrow{u} & {}\arrow{u} & {}\arrow{u} & {}
\end{tikzcd}
\end{equation}
where $(CH^*_{\red}(\mathcal{A},\mathcal{A}^{\vee},b^*)$ is the dual complex to the Hochschild complex. We obtain the reduced negative cyclic cochain complex $CC^*_{-,\red}(\mathcal{A},\mathcal{A}^{\vee})$ of $\mathcal{A}^{\vee}$ by taking the bicomplex consisting of the nonpositive columns of the above bicomplex. Similarly, the reduced positive cyclic cochain complex $CC^*_{+,\red}(\mathcal{A},\mathcal{A}^{\vee})$ is obtained by taking the bicomplex consisting of the positive columns of the bicomplex. The double complex consisting of all columns is called the periodic cyclic cochain complex, and it is denoted $CC^*_{\infty,\red}(\mathcal{A},\mathcal{A}^{\vee})$. 

The proof of~\cite[Lemma 3.6]{HL08} carries over to the filtered case to give a condition under which the connecting homomorphism is an isomorphism.
\begin{lemma}\label{hciso}
If there is an integer $N$ such that $HH_{\red}^*(\mathcal{A},\mathcal{A}^{\vee}) = 0$ whenever $*>N$, then for any integer $n$, there is an isomorphism
\[
HC^{n+1}_{+,\red}(\mathcal{A}) \xrightarrow{\sim} HC^n_{-,\red}(\mathcal{A}^{\vee}) \]
induced by $B^*$. \qed
\end{lemma}
The reduced Hochschild cohomology $HH_{\red}^*(\mathcal{A},\mathcal{A}^{\vee})$ is defined by taking the cohomology of the $R$-dual of the reduced Hochschild complex $CC_*^{\red}(\mathcal{A})$. 
\subsection{\texorpdfstring{$\infty$}{Infinity}-inner products and \texorpdfstring{$\infty$}{infinity}-cyclic potentials}
The following terminology is due to Tradler~\cite{Tra08}.
\begin{definition}
An $A_{\infty}$-bimodule homomorphism $\phi\colon\mathcal{A}_{\Delta}\to\mathcal{A}^{\vee}$ is called an \textit{$\infty$-inner product}. 
\end{definition}
For an uncurved $A_{\infty}$-algebra, one possible definition of a (weak proper) Calabi--Yau structure is a bimodule quasi-isomorphism $\mathcal{A}_{\Delta}\to\mathcal{A}^{\vee}$~\cite[Remark 49]{Gan23}. This data can, equivalently, be packaged as a negative cyclic cohomology class by Lemma~\ref{hciso}. Even for a curved $A_{\infty}$-algebra, one can obtain an $\infty$-inner product from a cocycle in $CC^*_{\red,-}(\mathcal{A},\mathcal{A}^{\vee})$.
\begin{lemma}[Cho-Lee~\cite{CL11}]\label{cocycletohom}
Consider a negative cyclic cocycle $\phi\in CC^*_{-,\red}(\mathcal{A},\mathcal{A}^{\vee})$, whose restriction to the $-i$th column of~\eqref{b*B*complex} is denoted $\phi_i$.
Then the sequence of maps
\begin{align*}
\phi_{p,q}\colon\mathcal{A}^{\otimes p}\otimes\underline{\mathcal{A}}\otimes\mathcal{A}^{\otimes q}\to\mathcal{A}^{\vee}
\end{align*}
defined by
\begin{align}
\phi_{p,q}(\alpha\otimes\underline{v}\otimes\beta)(w) = \psi_0(\alpha\otimes v\otimes\beta)(w)-\psi_0(\beta\otimes w\otimes\alpha)(v)
\end{align}
where
\begin{align*}
\alpha = a_1\otimes\cdots\otimes a_p\in\mathcal{A}^{\otimes p} \\
\beta = b_1\otimes\cdots\otimes b_{q}\in\mathcal{A}^{\otimes q}
\end{align*}
and $v,w\in\mathcal{A}$, is an $A_{\infty}$-bimodule homomorphism also denoted $\phi\colon\mathcal{A}_{\Delta}\to\mathcal{A}^{\vee}$. \qed
\end{lemma}
The proof that $\phi$ is a bimodule homomorphism uses the negative cyclic cocycle condition $b^*\phi_i = B^*\phi_{i+1}$, and this is the main connection between the trivialization of the $S^1$-action on the Fukaya category and our construction of the open Gromov--Witten potential. This particular chain-level correspondence between (negative) cyclic cocycles and bimodule homomorphisms is convenient, because the homomorphisms $\phi$ constructed this way admit useful symmetries.
\begin{corollary}
The bimodule homomomorphisms $\phi\colon\mathcal{A}_{\Delta}\to\mathcal{A}^{\vee}$ of Lemma~\ref{cocycletohom} are skew-symmetric and closed, meaning, respectively, that
\begin{itemize}
\item[•] for $\alpha$, $\beta$, $v$, and $w$ as in the statement of Lemma~\ref{cocycletohom}, we have that
\begin{align}
\phi_{p,q}(\alpha\otimes\underline{v}\otimes\beta)(w) = (-1)^{\kappa}\phi_{q,p}(\beta,w,\alpha)(v)
\end{align}
where
\[\kappa = \left(\sum_{i=1}^p |a_i|'+|v|'\right)\cdot\left(\sum_{j=1}^q|b_j|'+|w|'\right)\]
and;
\item[•] for $a_1\otimes\cdots\otimes a_{\ell+1}\in\mathcal{A}^{\otimes\ell+1}$ and any triple $1\leq i<j<k\leq\ell+1$, we have that
\begin{align}
&(-1)^{\kappa_i}\phi(\cdots\otimes\underline{a_i}\otimes\cdots)(a_j)+(-1)^{\kappa_j}\phi(\cdots\otimes\underline{a_j}\otimes\cdots)(a_k) \nonumber \\
&+(-1)^{\kappa_k}\phi(\cdots\otimes\underline{a_k}\otimes\cdots\otimes)(a_i) = 0
\end{align}
where the sign is determined by
\[ \kappa_* = \left(|a_1|'+\cdots+|a_*|'\right)\cdot\left(|a_{*+1}|'+\cdots+|a_k|'\right)\]
and where the inputs are cyclically ordered. \qed
\end{itemize}
\end{corollary}
\begin{remark}
A cyclic pairing on $\mathcal{A}$ can be thought of as a closed skew-symmetric $\infty$-inner product $\psi\colon\mathcal{A}_{\Delta}\to\mathcal{A}^{\vee}$ for which $\psi_{p,q} = 0$ whenever $p>0$ or $q>0$.
\end{remark}
We can relate the $\infty$-inner product obtained from a negative cyclic cocycle $\phi$ to the trace associated to $B^*\phi$.
\begin{lemma}\label{homvstrace}
For an $\infty$-inner product obtained via Lemma~\ref{homvstrace}, we have the identity
\begin{align}
\phi_0(1\otimes\mathfrak{m}_2(a_1,a_2)) = \phi_{0,0}(\underline{a_1})(a_2).
\end{align}
\end{lemma}
\begin{proof}
We have that $b^*\phi(1,a_1,a_2) = B^*\psi(1,a_1,a_2) = 0$ for a Hochschild cochain $\psi$ because $\phi$ is a negative cyclic cocycle. A direct calculation shows that
\begin{align*}
0&=b^*\phi_0(1\otimes a_1\otimes a_2)  \\
&= \phi_0(\mathfrak{m}_2(1\otimes a_1)\otimes a_2)-\phi_0(1\otimes\mathfrak{m}_2(a_1\otimes a_2))+(-1)^{|a_2|'\cdot(|a_1|'+1)}\phi_0(\mathfrak{m}_2(a_2\otimes 1)\otimes a_1)  \\
&= \phi_0(a_1\otimes a_2)-\phi_0(1\otimes\mathfrak{m}_2(a_1\otimes a_2))+(-1)^{|a_1|'\cdot|a_2|'+|a_2|'+|a_2|}\phi_0(a_2\otimes a_1).
\end{align*}
\end{proof}
We can think of the expression defined in Lemma~\ref{homvstrace} as the formula for a cyclic pairing on a canonical model of $\mathcal{A}$, in analogy in Kontsevich--Soibelman's theorem in the unfiltered setting. An $\infty$-inner product $\phi$ is said to be \textit{homologically nondegenerate} if for any nonzero $[a_1]\in H^*(\mathcal{A},\mathfrak{m}_{1,0})$ there is an element $[a_2]\in H^*(\mathcal{A},\mathfrak{m}_{1,0})$ such that $\phi_{0,0}(\underline{a_1})(a_2)$ on the chain level, for some representatives of these classes.
\begin{lemma}[Cho--Lee~\cite{CL11}]\label{KSanalogue}
Let $\phi$ be a closed skew-symmetric homologically nondegenerate $\infty$-inner product on $\mathcal{A}$. Then there is a canonical model for $\mathcal{A}$ carrying a strictly cyclic pairing such that we have a commutative diagram
\[
\begin{tikzcd}
\mathcal{A}\arrow{d}{\phi} & H^*(\mathcal{A},\mathfrak{m}_{1,0})\arrow{l}\arrow{d} \\
\mathcal{A}^*\arrow{r} & H^*(\mathcal{A},\mathfrak{m}_{1,0})^*
\end{tikzcd}
\]
where the right arrow comes from the strictly cyclic pairing, and the top arrow is a quasi-isomomorphism. \qed
\end{lemma}
Any $\phi$ which is both skew-symmetric and closed satisfies a weak analogue of cyclic symmetry.
\begin{lemma}[Cho--Lee~\cite{CL}]\label{weakcyclic}
Let $b,y\in\mathcal{A}$ and suppose that $|b| = 1$. Then for any $N\geq0$, we have that
\begin{align}
&N\sum_{p+q+k = N}\phi_{p,q}(b^{\otimes p}\otimes\underline{\mathfrak{m}_k(b^{\otimes k})}\otimes b^{\otimes q})(y) \\
&=\sum_{\substack{p+q+k = N \\ r+s = k-1}}\phi_{p,q}(b^{\otimes p}\otimes\underline{\mathfrak{m}_k(b^{\otimes r}\otimes y\otimes b^{\otimes s})}\otimes b^{\otimes q})(b) \\
&+\sum_{\substack{p+q+k=N \\ r+s  = p-1}}\phi_{p,q}(b^{\otimes r}\otimes y\otimes b^{\otimes s}\otimes\underline{\mathfrak{m}_k(b^{\otimes k})}\otimes b^{\otimes q})(b) \\
&+\sum_{\substack{p+q+k = N \\ r+s = q-1}}\phi_{p,q}(b^{\otimes p}\otimes\underline{\mathfrak{m}_k(b^{\otimes k})}\otimes b^{\otimes r}\otimes y\otimes b^{\otimes s})(b).
\end{align}
\qed
\end{lemma}
The potential of a cyclic $A_{\infty}$-algebra is defined as follows.
\begin{definition}\label{inftycyclicpotential}
If $\phi\colon\mathcal{A}_{\Delta}\to\mathcal{A}^{\vee}$ is an $\infty$-inner product, then the $\infty$-cyclic potential $\Phi'\colon F_{>0}\mathcal{A}\to R$ is a function on the set of elements of $\mathcal{A}$ of positive valuation defined by
\begin{align}
\Phi'(x)\coloneqq\sum_{N=0}^{\infty}\sum_{p+q+k = N}\frac{1}{N+1}\phi_{p,q}(x^{\otimes p}\otimes\underline{\mathfrak{m}_k(x^{\otimes k})}\otimes x^{\otimes q})(x). \label{potential}
\end{align}
\end{definition}
The sum in~\eqref{potential} converges since $\mathcal{A}$ is gapped and $b$ has positive valuation. 

Although it is not strictly necessary (cf. Theorem~\ref{truemain}), we can use Lemma~\ref{weakcyclic} to show that $\Phi'$ respects gauge-equivalence classes of bounding cochains. Recall that~\cite[Chapter 4.2]{FOOO} constructs, for any $A_{\infty}$-algebra $\mathcal{A}$ over a field of characteristic $0$, a model for the cylinder over $\mathcal{A}$ denoted $Poly([0,1],\mathcal{A})$ whose elements are pairs of formal polynomials in the Novikov variable $T$ with coefficients that are functions on the interval $[0,1]$. A consequence (cf.~\cite[Proposition 4.3.5]{FOOO} and~\cite[Lemma 4.3.7]{FOOO}) of this construction is that a pair of bounding cochains $b_0,b_1\in\widehat{\mathcal{M}}(\mathcal{A})$ are gauge-equivalent if and only if there is a path of elements
\[ b_t = \sum_i b_i(t) T^{\lambda_i}\in\mathcal{A} \]
with $\lim\lambda_i = \infty$ such that
\begin{itemize}
\item[•] $b_i(t)$ is a polynomial in the variable $t$; and
\item[•] for each fixed $t$, the element $b_t\in\mathcal{A}$ is a bounding cochain.
\end{itemize}
\begin{theorem}
For any gauge-equivalent bounding cochains $b_0,b_1\in\widehat{\mathcal{M}}(\mathcal{A})$, one has that $\Phi'(b_0) = \Phi'(b_1)$.
\end{theorem}
\begin{proof}
Choosing a path $b_t$ as above, we compute
\begin{align*}
\frac{d}{dt}\Phi'(b_t) &= \sum_{N=0}^{\infty}\sum_{p+q+k = N}\frac{1}{N+1}\phi_{p,q}\left(b_t^{\otimes p}\otimes\underline{\mathfrak{m}_k(b_t^{\otimes k})}\otimes b_t^{\otimes q}\right)\left(\frac{db_t}{dt}\right) \\
&+\sum_{N=0}^{\infty}\sum_{\substack{p+q+k = N \\ r+s = k-1}}\frac{1}{N+1}\phi_{p,q}\left(b_t^{\otimes p}\otimes\underline{\mathfrak{m}_k\left(b_t^{\otimes r}\otimes\frac{db_t}{dt}\otimes b_t{\otimes s}\right)}\otimes b_t^{\otimes q}\right)(b_t) \\
&+\sum_{N=0}^{\infty}\sum_{\substack{p+q+k = N \\ r+s = p-1}}\frac{1}{N+1}\phi_{p,q}\left(b_t^{\otimes r}\otimes\frac{db_t}{dt}\otimes b_t^{\otimes s}\otimes\underline{\mathfrak{m}_k(b_t^{\otimes k})}\otimes b_t^{\otimes q}\right)(b_t) \\
&+\sum_{N=0}^{\infty}\sum_{\substack{p+q+k = N \\ r+s = q-1}}\frac{1}{N+1}\phi_{p,q}\left(b_t^{\otimes p}\otimes\underline{\mathfrak{m}_k(b_t^{\otimes k})}\otimes b_t^{\otimes r}\otimes\frac{db_t}{dt}\otimes b_t^{\otimes s}\right)(b_t) \\
&= \sum_{N=0}^{\infty}\sum_{p+q+k = N}\phi_{p,q}\left(b_t^{\otimes p}\otimes\underline{\mathfrak{m}_k(b_t^{\otimes k})}\otimes b_t^{\otimes q}\right)\left(\frac{db_t}{dt}\right) = 0
\end{align*}
where the second equality follows from Lemma~\ref{weakcyclic}, and the last equality follows from the Maurer--Cartan equation.
\end{proof}

\section{Lagrangian Floer theory}\label{lagfloerthy} The purpose of this section is to review the Morse--theoretic model for the Lagrangian Floer cochcain complex. In this discussion, we fix notation for pseudoholomorphic pearly trees that will be helpful when we construct the cyclic open-closed map. We will also explain how to construct $A_{\infty}$-structures on the Morse complex of a cylinder, which we  need to study the invariance of the open Gromov--Witten potential.

\subsection{Pseudoholomorphic pearly trees}\label{plainfloer}
Suppose that $M = (M^{2n},\omega)$ is a closed connected symplectic manifold. Let $\mathcal{J}(M)$ denote the space of $\omega$-tame almost compatible structures on $M$, and let $J\in\mathcal{J}(M)$. For the rest of this section, fix a closed connected Lagrangian embedding $L\subset M$, where $L$ is equipped with a spin structure $\mathfrak{s}$ and a $GL(1,\Bbbk)$ local system. Additionally, choose Morse--Smale pairs $(f_L,g_L)$ and $(f_M,g_M)$ on $L$ and $M$, respectively. The sets of critical points of $f_L$ and $f_M$ are denoted $\crit(f_L)$ and $\crit(f_M)$.
\begin{assumption}\label{uniqueminmax}
The Morse functions $f_L$ and $f_M$ both have a unique local minimum and a unique local maximum.
\end{assumption}
We can define the Morse cochain complexes $(CM^*(L;R),d)$ and $(CM^*(M;Q),d)$ in terms of these Morse--Smale pairs, with coefficients in $R$ and $Q$, respectively, whose differentials count isolated gradient flow lines joining two critical points.

Given a class $\beta\in H_2(M,L;\mathbb{Z})$ and nonnegative integers $k,\ell\geq0$, consider the (uncompactified) moduli space $\mathcal{M}_{k+1,\ell}(L;\beta)$ of all $J$-holomorphic disks $u\colon(D^2,\partial D^2)\to(M,L)$ with $k+1$ cyclically ordered boundary marked points $z_0,\ldots,z_k$ and $\ell$ ordered interior marked points $w_1,\ldots,w_{\ell}$. Similarly, for $\beta\in H_2(M;\mathbb{Z})$, let $\mathcal{M}_{\ell}(\beta)$ denote the (uncompactified) moduli space of $J$-holomorphic spheres in $M$ with $\ell$ marked points denoted $w_1,\ldots,w_{\ell}$. Define the boundary evaluation maps
\begin{align*}
\evb_j^{\beta}\colon\mathcal{M}_{k+1,\ell}(L;\beta)&\to L\\
\evb_j^{\beta}(u) &= u(z_j)
\end{align*}
for $j = 0,\ldots,k$. Similarly, define the interior evaluation maps by
\begin{align*}
\evi_j^{\beta}\colon\mathcal{M}_{k+1,\ell}(L;\beta)&\to M \\
\evi_j^{\beta}(u) &= u(w_j)
\end{align*}
There are also interior evaluation maps $\evi^{\beta}_j\colon\mathcal{M}_{\ell}(\beta)\to M$ defined similarly.

The $A_{\infty}$-operations on $CM^*(L;R)$ will count configurations consisting of trees of elements in the moduli spaces $\mathcal{M}_{k+1,\ell}(\beta)$ joined by gradient flow lines of (perturbations of) $f_L$. The combinatorial structures underlying such configurations are described by oriented metric ribbon trees whose vertices are partitioned depending on whether they parametrize sphere or disk components of a pearly tree.
\begin{definition}\label{bicoloredtree}
A \textit{bicolored} tree is a tree $T$ with vertex set $V(T)$ and edge set $E(T)$, together with a partition of its vertices
\[ V(T) = V_{\circ}(T)\sqcup V_{\bullet}(T) \]
called the disk vertices and sphere vertices of $T$, respectively. We require that $T$ come equipped with a choice of a subtree $T_{\circ}$ whose vertex set $V(T_{\circ})$ coincides with $V_{\circ}(T)$. Let $E_{\circ}(T)\coloneqq E(T_{\circ})$ and $E_{\bullet}(T)\coloneqq E(T)\setminus E_{\circ}(T)$. Finally, let $e_{0}^{\circ},\ldots,e_{k}^{\circ}$ denote the (combinatorially) semi-infinite edges contained in $T_{\circ}$, and let $e_{1}^{\bullet},\ldots,e_{\ell}^{\bullet}$ denote the remaining semi-infinite edges.
\end{definition}
In the above we also allow the exceptional case of a tree $T$ with $V(T) = \emptyset$ consisting of a single infinite edge.
\begin{definition}
An \textit{oriented metric ribbon tree} consists of a bicolored tree $T$, equipped with
\begin{itemize}
\item[•] a ribbon structure on $T_{\circ}$; i.e. a cyclic ordering of all edges in $E_{\circ}(T)$ adjacent to any vertex in $T_{\circ}$ (which induces a cyclic ordering of $e_{0}^{\circ},\ldots,e_{k}^{\circ}$) and an ordering of the all edges in $E_{\bullet}(T)$ adjacent to any vertex of $T$;
\item[•] a metric on $T$, which is described by a length function $\lambda\colon E(T)\to\mathbb{R}_{\geq0}$;
\item[•] an orientation on $T$ determined by orienting $e^{\circ}_0$ so that it is an outgoing edge, and orienting all remaining edges of $T$ so that they point toward $e^{\circ}_0$;
\item[•] a class $\beta_v\in H_2(M,L)$ for each $v\in V_{\circ}(T)$ and a class $\beta_v\in H_2(M)$ for each $v\in V_{\bullet}(T)$.
\end{itemize}
\end{definition}
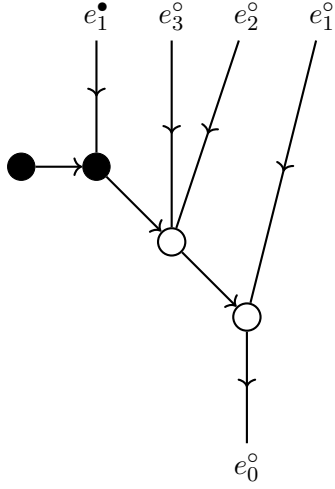
\begin{figure}
\begin{tikzpicture}
\begin{scope}[every node/.style={circle,thick,draw}]
	\node (C) at (2,-1) {};
	\node (D) at (3,-2) {};
\end{scope}

\begin{scope}[every node/.style={circle,draw,fill = black}]
	\node (A) at (0,0) {};
	\node (B) at (1,0) {};
\end{scope}

\node (E) at (1,2) {$e_1^{\bullet}$};
\node (F) at (2,2) {$e_3^{\circ}$};
\node (G) at (3,2) {$e_2^{\circ}$};
\node (H) at (4,2) {$e_1^{\circ}$};
\node (I) at (3,-4) {$e_0^{\circ}$};
\path[->, thick] (C) edge node {} (D);
\path[->, thick] (A) edge node {} (B);
\path[->, thick] (B) edge node {} (C);

\begin{scope}[very thick,decoration={
	markings,
	mark=at position 0.5 with {\arrow{>}}}]

\draw[postaction={decorate}, thick] (E) -- (B);
\draw[postaction={decorate}, thick] (F) -- (C);
\draw[postaction={decorate}, thick] (G) -- (C);
\draw[postaction={decorate}, thick] (H) -- (D);
\draw[postaction={decorate}, thick] (D) -- (I);

\end{scope}
\end{tikzpicture}\caption{An oriented metric ribbon tree.}\label{ribbontreepic}
\end{figure}
For any $v\in V(T)$, let $\val_{\circ}(v)$ denote the number of edges in $E_{\circ}(T)$ adjacent to $v$ and $\val_{\bullet}(v)$ denote the number of edges in $E_{\bullet}(T)$ adjacent to $v$. 
\begin{definition}
We say that $T$ is stable if for each $v\in V(T)$ for which $\omega(\beta_v)=0$, either
\begin{itemize}
\item[•] $v\in V_{\circ}(T)$ and $\val_{\circ}(v)+2\cdot\val_{\bullet}(v)\geq3$; or
\item[•] $v\in V_{\bullet}(T)$ and $\val_{\bullet}(v)\geq3$.
\end{itemize}
\end{definition}
From now on, we will only ever consider moduli spaces of stable trees, both in this construction and in all others like it. There is a moduli space of stable oriented metric ribbon trees, which can be compactified by allowing the length an edge to go to infinity and break. Given two oriented metric ribbon trees $T_1$ and $T_2$, we can attach endpoints to the edge $e^{\circ}_0$ of $T_1$ and to the edge $e^{\circ}_i$, for some $i>0$, of $T_2$, and glue the two endpoints together to form a new tree. Since we have glued the output edge of $T_1$ to an input edge of $T_2$, the glued tree carries a ribbon structure and orientation.

We associate to each vertex of $T$ the moduli space $\mathcal{M}_{\val_{\circ}(v)-1,\val_{\bullet}(v)}(\beta_v)$ if $v\in E_{\circ}(T)$, or $\mathcal{M}_{\val_{\bullet}(v)}(\beta_v)$ if $v\in E_{\bullet}(T)$. For brevity, we write $\mathcal{M}(\beta_v)$ for either of these moduli spaces. Let $E^f_{\circ/\bullet}(T)$ denote the sets of combinatorially finite edges of $E_{\circ/\bullet}(T)$. We will construct an evaluation map 
\begin{align*}
\ev_{T^f}\colon\prod_{v\in V(T)}\mathcal{M}(\beta_v)\to\prod_{e\in E_{\circ}^f(T)}(L\times L)\times\prod_{e\in E_{\bullet}^f(T)}(M\times M)
\end{align*}
using the bicoloring of $T$.

If $e\in E(T)$ is a combinatorially finite edge, let $s(e),t(e)\in V(T)$ denote the source and target of $e$, respectively. When $e\in E_{\circ}(T)$, there is an integer $k_t$ such that $e$ comes $k_t$th in the cyclic ordering of edges adjacent to $t(e)$. Our orientation conventions imply that $e$ is always the zeroth edge of $s(e)$. Similarly, when $e\in E_{\bullet}(T)$, there are integers $k_s$ and $k_t$ which are defined analogously using the orderings of the edges in $E_{\bullet}(T)$.

Let $\vec{u} = (u_v)_{v\in V(T)}$ denote an element of $\prod_{v\in V(T)}\mathcal{M}(\beta_v)$. For $e\in E_{\circ}(T)$, define
\begin{align*}
&\ev_e\colon\prod_{v\in V(T)}\mathcal{M}(\beta_v)\to L\times L \\
&\ev_e(\vec{u}) = (\evb_0(u_{s(e)}),\evb_{k_t}(u_{t(e)})).
\end{align*}
For $e\in E_{\bullet}(T)$, define
\begin{align*}
&\ev_e\colon\prod_{v\in V(T)}\mathcal{M}(\beta_v)\to M\times M \\
&\ev_e(\vec{u}) = (\evi_{k_s}(u_{s(e)}),\evi_{k_t}(u_{t(e)})).
\end{align*}
Finally, set
\begin{align*}
\ev_{T^f}(\vec{u}) = \prod_{e\in E_{\circ}^f(T)}\ev_e(\vec{u})\times\prod_{e\in E_{\bullet}^f(T)}\ev_e(\vec{u}).
\end{align*}
We extend this to a full evaluation map for $T$ by taking into account the semi-infinite edges. According to our orientation conventions, the edge $e_0^{\circ}$ has one endpoint denoted $s(e_0^{\circ})$, and all other semi-infinite edges $e_j^{\circ/\bullet}$, where $j = 	1,\ldots,k$ or $j = 1,\ldots,\ell$, have one endpoint denoted $t(e_j^{\circ/\bullet})$. With this understood, define the evaluation maps $\ev_j^{\circ/\bullet}$ to be the evaluation map determined by the position of $e_j^{\circ/\bullet}$ in the ordering of edges adjacent to its endpoint. We can now associate to $T$ an evaluation map
\begin{align}
\ev_T(\vec{u}) = \prod_{j=1}^{\ell}\ev_{j}^{\bullet}(u_{t(e_j^{\bullet})})\times\prod_{j=1}^k\ev_j^{\circ}(u_{t(e_j^{\circ})})\times\ev_{T^f}(\vec{u})\times \ev_0^{\circ}(u_{s(e_0^{\circ})}). \label{evtree}
\end{align}
Having defined $\ev_T$, we will define the moduli spaces of pearly trees by pulling back a submanifold in the codomain of this map. We must also assign to each edge of $T$ a Morse function of the following type.
\begin{definition}
Fix a Morse function $f_0\colon Y\to\mathbb{R}$ on a compact manifold $Y$. We say that a Morse function $f\colon Y\to\mathbb{R}$ is \textit{$f_0$-admissible} if it is a $C^2$-small perturbation of $f_0$ for which $\crit(f_0) = \crit(f)$ and which agrees with $f_0$ in a neighborhood of its critical points.
\end{definition}
For each $e\in E_{\circ}(T)$, choose an $f_L$-admissible Morse function $f_{L,e}$ and for each $e\in E_{\bullet}(T)$, choose an $f_M$-admissible Morse function $f_{M,e}$. Given a combinatorially finite edge $e\in E_{\circ}^f(T)$, let $\phi_t^{f_{L,e}}$ denote the time-$t$ gradient flow of $f_{L,e}$, and for $e\in E_{\bullet}^f(T)$, let $\phi_t^{f_{M,e}}$ denote the time-$t$ gradient flow of $f_{M,e}$. These yield embeddings
\begin{align}
(L\setminus\crit(f_L))\times\mathbb{R}_{\geq0}&\hookrightarrow L\times L \nonumber \\
(x,t)&\mapsto(x,\phi_t^{f_{L,e}}(x)) \label{graphcirc}
\end{align}
and
\begin{align}
(M\setminus\crit(f_M))\times\mathbb{R}_{\geq0}&\hookrightarrow M\times M \nonumber \\
(x,t)&\mapsto(x,\phi_t^{f_{M,e}}(x)) \label{graphbullet}
\end{align}
whose images are denoted $G_e$. Given $x\in\crit(f_L)$, let $W^u(x)$ and $W^s(x)$ denote its unstable and stable manifolds, and define $W^u(y)$ and $W^s(y)$ for $y\in\crit(f_M)$ analogously.
\begin{definition}
Let $x = (x_1,\ldots,x_k)$ be a sequence of inputs in $\crit(f_L)$ and $y = (y_1,\ldots,y_{\ell})$ be a sequence of inputs in $\crit(f_M)$, and let $x_0\in\crit(f_L)$ denote an output critical point. Define the \textit{moduli space of pearly trees in the class} $\beta\in H_2(M,L)$ with these inputs and output to be
\begin{align*}
\mathcal{M}(x_0,x;y;\beta)\coloneqq\coprod_T\ev_T^{-1}\left(\prod_{j=1}^{\ell}W^u(y_j)\times\prod_{j=1}^k W^u(x_j)\times\prod_{e\in E^f_{\circ}(T)\sqcup E^f_{\bullet}(T)}G_e\times W^s(x_0)\right)
\end{align*}
where the disjoint union is taken over the set of all oriented metric ribbon trees $T$ with $\sum_{v\in V(T)}\beta_v = \beta$. This admits a natural Gromov compactification
\begin{align}
\overline{\mathcal{M}}(x_0,x;y;\beta).\label{pearls}
\end{align}
\end{definition}
\begin{figure}
\begin{tikzpicture}
\begin{scope}[very thick,decoration={
	markings,
	mark=at position 0.5 with {\arrow{>}}}]

\begin{scope}[scale = 0.5]
\draw[postaction = {decorate}] (2,0) -- (5,-3);
\draw[postaction = {decorate}] (5,3) -- (5,-3);
\draw[postaction = {decorate}, rounded corners] (8,3) -- (8,0) -- (5,-3);
\draw[postaction = {decorate}] (5,-3) -- (8,-6);
\draw[postaction = {decorate}, rounded corners] (11,3) -- (11,-3) -- (8,-6);
\draw[postaction = {decorate}] (8,-6) -- (8,-9);
\draw[postaction = {decorate}] (2,3) -- (2,0);
\draw[fill = gray!40] (5,-3) circle (0.75cm);
\draw[fill = gray!40] (8,-6) circle (0.75cm);

\node[anchor = south] at (2,3) {$y_1$};
\node[anchor = south] at (5,3) {$x_3$};
\node[anchor = south] at (8,3) {$x_2$};
\node[anchor = south] at (11,3) {$x_1$};
\node[anchor = north] at (8,-9) {$x_0$};

\begin{scope}
\shade[ball color = gray!40, opacity = 0.4] (0,0) circle (1cm);
\draw (0,0) circle (1cm);
\draw (-1,0) arc (180:360:1 and 0.3);
\draw[dashed] (1,0) arc (0:180:1 and 0.3);
\end{scope}

\begin{scope}[xshift = 2cm]
\draw[fill = white] (0,0) circle (1cm);
\shade[ball color = gray!40, opacity = 0.4] (0,0) circle (1cm);
\draw (0,0) circle (1cm);
\draw (-1,0) arc (180:360:1 and 0.3);
\draw[dashed] (1,0) arc (0:180:1 and 0.3);
\end{scope}
\end{scope}
\end{scope}
\end{tikzpicture}\caption{A pseudoholomorphic pearly tree with underlying oriented metric ribbon tree depicted in Figure~\ref{ribbontreepic}.}
\end{figure}
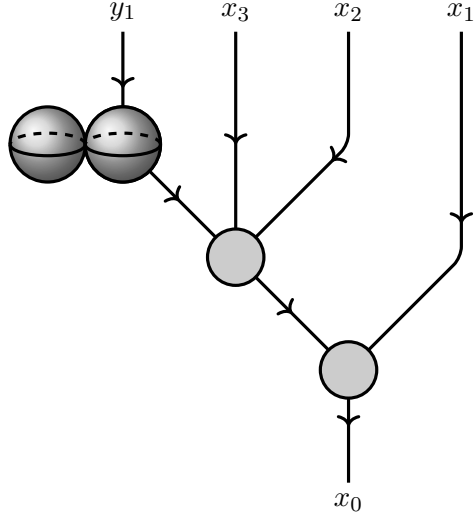

\begin{remark}
Our definitions differ from those in~\cite{CW22} or~\cite{VWX20} in that we have not used domain-dependent almost complex structures. This slightly simplifies the construction of a strict unit and later of the cyclic open-closed map.
\end{remark}
We assume that $J$ can be chosen so that these moduli spaces are regular.
\begin{assumption}\label{regularity}
There exists $J\in\mathcal{J}(M)$ such that all of the moduli spaces~\eqref{pearls} of virtual dimension at most $1$ are compact orbifolds with boundary of the expected dimension.
\end{assumption}
Since we have assumed that $L$ is spin, the moduli spaces~\eqref{pearls} are oriented.
\begin{definition}
Given sequences $x$ and $y$ as above, define
\[ \mathfrak{q}_{k,\ell}^{\beta}(x_1,\ldots,x_k;y_1,\ldots,y_{\ell}) = \sum_{x_0\in\crit(f_L)}\hol_{\nabla}(\beta)\#|\mathcal{M}(x_0,x;y;\beta)|\cdot x_0\]
where $\#|\mathcal{M}(x_0,x;y;\beta)|$ is the signed count of elements in the zero-dimensional moduli space $\mathcal{M}(x_0,x;y;\beta)$.

These extend linearly to arbitrary inputs in $CM^*(L;R)$ and $CM^*(M;Q)$. Define operations
\begin{align*}
&\mathfrak{q}_{k,\ell}\colon CM^*(L;R)^{\otimes k}\otimes CM^*(M:Q)^{\otimes\ell}\to CM^*(L;R) \\
&\mathfrak{q}_{k,\ell} = \sum_{\beta\in H_2(M,L;\mathbb{Z})} e^{\mu(\beta)/2} T^{\omega(\beta)}\mathfrak{q}^{\beta}_{k,\ell}.
\end{align*}
For $\gamma\in CM^*(L;R)$ with $d\gamma = 0$ and $|\gamma| = 2$, where $d$ denotes the differential on the Morse cochain complex, define the bulk-deformed operators
\begin{align*}
\mathfrak{m}_{k}^{\gamma}(x) = \sum_{\ell\geq0}\frac{1}{\ell!}\mathfrak{q}_{k,\ell}(x;\gamma^{\otimes\ell}).
\end{align*}
\end{definition}
We call such a class $\gamma$ a \textit{bulk parameter}. The next lemma follows from Assumption~\ref{regularity}, by standard arguments.
\begin{lemma}
For any $\gamma\in CM^*(M)$ with $|\gamma| = 2$ and $d\gamma=0$, the pair
\[ (CM^*(L),\lbrace\mathfrak{m}_{k}^{\gamma}\rbrace_{k=0}^{\infty})\]
is a strictly unital gapped filtered $A_{\infty}$-algebra. The unit element is given by the unique minimum of $f_L$. \qed
\end{lemma}
For a discussion of the existence of a unit, see the proof of~\cite[Lemma 5.2.4]{BC07}.
\subsection{Models for cylinder objects}\label{cylinderobjects}
In~\cite{F11} and~\cite{ST21}, the behavior of the open Gromov--Witten potential as the almost complex structure on $M$ varies is understood in terms of pseudo-isotopies, which can be viewed as $A_{\infty}$-structures on the space of differentials forms on $L\times[-1,1]$. In this subsection, we develop the analogous notion in the pearly model.

Suppose that we are given two almost complex structures $J_{\pm 1}\in\mathcal{J}(M)$, two Morse--Smale pairs $(f_L^{\pm 1},g_L^{\pm 1})$ on $L$, and two Morse--Smale pairs on $(f_M^{\pm 1},g_M^{\pm 1})$ on $M$, both of which satisfy Assumption~\ref{regularity}. Consider two smooth functions 
\begin{align*}
&F_L\colon L\times(-1-\epsilon,1+\epsilon)\to\mathbb{R}\\ &F_M\colon M\times(-1-\epsilon,1+\epsilon)\to\mathbb{R}
\end{align*}
where, if $F^t_L(x)\coloneqq F_L(x,t)$ and $F^t_M(x)\coloneqq F_M(x,t)$, then for some $\epsilon>0$, we have that
\begin{align*}
F_L^t = \begin{cases}
f^{-1}_L \, , & t\in(-1-\epsilon,-1+\epsilon] \\
f^1_L \, , & t\in[1-\epsilon,1+\epsilon)
\end{cases}
\end{align*}
and similarly for $F_M$. We can also assume without loss of generality that $F_L^t$ and $F_M^t$ are independent of $t$ and are Morse functions $f^0_L\colon L\to\mathbb{R}$ and $f^0_M\colon M\to\mathbb{R}$, respectively, provided that $t\in(-\epsilon,\epsilon)$. We modify $F_L$ and $F_M$ to obtain Morse functions on $L\times[-1,1]$ and $M\times[-1,1]$ by choosing a Morse function $\rho\colon(-1-\epsilon,1+\epsilon)\to\mathbb{R}$ such that
\begin{itemize}
\item $\rho$ has index $0$ critical points at $\pm 1$ and an index $1$ critical point at $0$;
\item $\rho$ is sufficiently increasing on $(-1,0)$ and sufficiently decreasing on $(0,1)$ that
\begin{align*}
\begin{cases}
\frac{\partial F_M}{\partial t}(y,t)+\rho'(t)>0 \text{ and }\frac{\partial F_L}{\partial t}(x,t)+\rho'(t)>0 \, , & t\in(-1,0) \, , \, y\in M \, , \, x\in L \\
\frac{\partial F_M}{\partial t}(y,t)+\rho'(t)<0 \text{ and }\frac{\partial F_L}{\partial t}(x,t)+\rho'(t)<0 \, , & t\in(0,1) \, , \, y\in M \, , \, x\in L
\end{cases}
\end{align*}
\item $F^t_L(x)+\rho(t)$ and $F^t_M(y)+\rho(t)$ are Morse functions on $L$ and $M$ for all $t\in(-\epsilon,\epsilon)$.
\end{itemize}
It follows from these conditions on $\rho$ that the functions
\begin{align*}
\widetilde{f}_L(x,t)\coloneqq F_L(x,t)+\rho(t)\colon L\times(-1-\epsilon,1+\epsilon)\to\mathbb{R} \\
\widetilde{f}_M(y,t)\coloneqq F_M(y,t)+\rho(t)\colon M\times(-1-\epsilon,1+\epsilon)\to\mathbb{R}
\end{align*}
are Morse functions whose critical point sets are
\begin{align*}
\crit(\widetilde{f}_L) &= \crit(f_L^{-1})\times\lbrace-1\rbrace\cup\crit(f_L^0)\times\lbrace 0\rbrace\cup\crit(f_L^1)\times\lbrace 1\rbrace \\
\crit(\widetilde{f}_M) &= \crit(f_M^{-1})\times\lbrace-1\rbrace\cup\crit(f_M^0)\times\lbrace 0\rbrace\cup\crit(f_M^1)\times\lbrace 1\rbrace
\end{align*}
with Morse indices given by
\begin{align*}
\ind_{F_L}(x,\pm1) &= \ind_{f_L^{\pm 1}}(x) \\
\ind_{F_L}(x,0) &= \ind_{f_L^0}(x)+1 \\
\ind_{F_M}(y,\pm1) &= \ind_{f_M^{\pm 1}}(y) \\
\ind_{F_M}(y,0) &= \ind_{f_M^0}(y)+1
\end{align*}
By a partition of unity argument, one can construct a Riemannian metric $\widetilde{g}_L$ on $L\times[-1,1]$ such that
\begin{itemize}
\item[•] the restrictions of $\widetilde{g}_L$ to $L\times[-1,-1+\epsilon)$ and $L\times(1-\epsilon,1]$ agree with the products of the metrics $g_L^{\mp 1}$ with the standard metric on the interval;
\item[•] for all $t\in(-\epsilon,\epsilon)$, the restriction of $\widetilde{g}_L$ to $L\times\lbrace t\rbrace\cong L$ is a Riemannian metric $g_L^0$ on $L$ which is independent of $t$, and $(f^0_L,g^0_L)$ is a Morse--Smale pair;
\item[•] $(\widetilde{f}_L,\widetilde{g}_L)$ is a Morse--Smale pair on $L\times[-1,1]$.
\end{itemize}
We can repeat this construction with $M$ in place of $L$ to obtain a Morse--Smale pair $(\widetilde{f}_M,\widetilde{g}_M)$. Define the Morse cochain complexes $(CF^*(L\times[-1,1]),d)$ and $(CF^*(M\times[-1,1]),d)$ with respect to these Morse--Smale pairs.

To define $A_{\infty}$-operations on $CM^*(L\times[-1,1])$, we consider moduli spaces of disks defined using time-dependent almost complex structures. Let $\underline{J} = \lbrace J_t\rbrace_{t\in[-1,1]}$ be a path in $\mathcal{J}(M)$, with the property that the moduli spaces of pearly trees of virtual dimension at most $1$ defined using the Morse--Smale pairs $(f^i_L,g^i_L)$ and $(f^i_M,g^i_M)$ and the almost complex structures $J_i$, for $i\in\lbrace-1,0,1\rbrace$, satisfy Assumption~\ref{regularity}.

For any $\beta\in H_2(M,L;\mathbb{Z})$ or $\beta\in H_2(M;\mathbb{Z})$, define \textit{time-dependent moduli spaces}
\begin{align*} \widetilde{\mathcal{M}}_{k+1,\ell}(\beta)\coloneqq\lbrace(u,t)\colon u\in\mathcal{M}_{k+1,\ell}(\beta;J_t)\rbrace \\
\widetilde{\mathcal{M}}_{\ell}(\beta)\coloneqq\lbrace(u,t)\colon u\in\mathcal{M}_{\ell}(\beta;J_t)\rbrace
\end{align*}
Let $\evb_{j,t}^{\beta}\colon\mathcal{M}_{k+1,\ell}(\beta;J_t)\to L$ and $\evi_{j,t}^{\beta}\colon\mathcal{M}_{k+1,\ell}(\beta;J_t)\to M$ denote the boundary and interior evaluation maps at time $t$, respectively. The time-dependent moduli spaces carry boundary evaluation maps
\begin{align*}
\widetilde{\evb}_j^{\beta}\colon\widetilde{\mathcal{M}}_{k+1,\ell}(L;\beta)&\to L\times[-1,1]\\
\widetilde{\evb}_j^{\beta}(u,t) &= (\evb_{j,t}^{\beta}(u),t)
\end{align*}
for $j = 0,\ldots,k$. There are similarly defined interior evaluation maps
\begin{align*}
\widetilde{\evi}_j^{\beta}\colon\widetilde{\mathcal{M}}_{k+1,\ell}(L;\beta)&\to M\times[-1,1]\\
\widetilde{\evi}_j^{\beta}(u,t) &= (\evi_{j,t}^{\beta}(u),t) \\
\widetilde{\evi}_j^{\beta}\colon\widetilde{\mathcal{M}}_{\ell}(L;\beta)&\to M\times[-1,1]\\
\widetilde{\evi}_j^{\beta}(u,t) &= (\evi_{j,t}^{\beta}(u),t)
\end{align*}
for all $j = 0,\ldots,\ell$. These let us define analogues of the evaluation maps~\eqref{evtree}, enabling the following definition.
\begin{definition}
Let $\tilde{x} = (\tilde{x}_1,\ldots,\tilde{x}_k)$ be a sequence of inputs in $\crit(\widetilde{f}_L)$ and $\tilde{y} = (\tilde{y}_1,\ldots,\tilde{y}_{\ell})$ be a sequence of inputs in $\crit(\widetilde{f}_M)$, and let $\tilde{x}_0\in\crit(\widetilde{f}_L)$ be an output critical point. The compactified moduli spaces
\[ \overline{\mathcal{M}}(\tilde{x}_0,\tilde{x};\tilde{y};\beta) \label{cylpearls} \]
are defined the same way as the moduli spaces~\eqref{pearls}, except that, given an oriented metric ribbon tree $T$, 
\begin{itemize}
\item[•] the vertices $v\in V_{\circ}(T)$ are assigned elements of $\widetilde{\mathcal{M}}_{k+1,\ell}(\beta)$ and the vertices $v\in V_{\bullet}(T)$ are assigned elements of $\widetilde{\mathcal{M}}_{\ell}(\beta)$; and
\item[•] the edges of $T$ are assigned small generic perturbations of $\widetilde{f}_L$ or $\widetilde{f}_M$ which agree with $\widetilde{f}_L$ and $\widetilde{f}_M$, respectively, near the critical points.
\end{itemize}
\end{definition}
The time-dependent analogue of Assumption~\ref{regularity} is the following.
\begin{assumption}\label{cylregularity}
There is a path $\underline{J}$ joining $J_{-1}$ and $J_1$ such that all moduli spaces~\eqref{cylpearls} of virtual dimension $\leq1$ are compact orbifolds with boundary of the expected dimension.
\end{assumption}
\begin{remark}\label{cylsubmersivity}
In the de Rham or singular chains model for Lagrangian Floer theory, one would require submersivity the boundary evaluation maps on $\widetilde{\mathcal{M}}_{k+1,\ell}(\beta)$ to define $A_{\infty}$-structures on $L\times I$. If this is the case, then regularity of the moduli spaces would implies that all of the moduli spaces $\widetilde{\mathcal{M}}_{k+1,\ell}(\beta;J_t)$ satisfy Assumption~\ref{regularity}. The existence of such paths $\underline{J}$ cannot be established using standard transversality arguments.
\end{remark}
Under this assumption, we define bulk-deformed $A_{\infty}$-structures on $CM^*(L\times[-1,1];R)$ just as we did on $CM^*(L;R)$. Given a cocycle $\widetilde{\gamma}\in CM^*(M\times[-1,1])$, our choice of Morse function on $M\times[-1,1]$ implies that its restrictions to $M\times\lbrace\mp 1\rbrace$ are Morse cocycles  $\gamma^{\mp 1}\in CM^*(M;(f_M^{\mp 1},g_M^{\mp 1}))$ for which $|\gamma^{\mp 1}| = |\gamma|$.
\begin{lemma}
Let $\widetilde{\gamma}\in CM^*(M\times[-1,1])$ be a class with $|\widetilde{\gamma}| = 2$ and $d\widetilde{\gamma}=0$. Then there exist bulk-deformed operators
\[ \widetilde{\mathfrak{m}}_k^{\widetilde{\gamma}}\colon CM^*(L\times[-1,1];R)^{\otimes k}\to CM^*(L\times[-1,1];R)\]
of degree $2-k$ such that $(CM^*(L;R),\lbrace\widetilde{\mathfrak{m}}_k^{\widetilde{\gamma}}\rbrace)$ is a gapped filtered $A_{\infty}$-algebra.

Furthermore, this $A_{\infty}$-algebra has
\begin{align*}
(CM^*(L;(f_L^{\mp 1},g_L^{\mp 1})),\mathfrak{m}^{\gamma^{\mp 1}}_k)
\end{align*}
as $A_{\infty}$-subalgebras. \qed
\end{lemma}

\section{The cyclic open-closed map}\label{thecyclicocmap}
For the next two sections, let $\mathcal{A} = (CF^*(L),\lbrace\mathfrak{m}_k^{\gamma}\rbrace)$ denote the possibly bulk-deformed curved $A_{\infty}$-algebra constructed in Section~\ref{plainfloer}. Denote by $(CM_*(M),\partial)$ the Morse chain complex constructed using the Morse--Smale pair already chosen in the construction of $\mathcal{A}$. We will construct a sequence of maps
\begin{align*}
\mathcal{OC}_m\colon CH_*(\mathcal{A})\to CM_{n-*+2m}(M)
\end{align*}
where the map $\mathcal{OC}_0$ is the usual open-closed map. The main result of this section says that these can be assembled to give a chain map $\mathcal{OC}\colon CC_*^+(\mathcal{A})\to CM_*(M)$. It follows (cf. Assumption~\ref{uniqueminmax}) that this induces a class in $HC^{*}_{+,\red}(\mathcal{A})$, and an $\infty$-inner product on $\mathcal{A}$ in turn.
\begin{lemma}\label{equivariancelemma}
For each $m\geq0$, we have that
\begin{align}
\mathcal{OC}_{m-1}\circ B + \mathcal{OC}_m\circ b = 0. \label{equivariance}
\end{align}
\end{lemma}
To construct $\mathcal{OC}_m$, we need to consider, for $m,k,\ell\in\mathbb{Z}_{\geq0}$, two types of (uncompactified) moduli spaces of domains
\begin{align}
\mathcal{P}_{m,k,\ell} \label{ocmod}\\
\mathcal{P}_{m,k,\ell}^{S^1} \label{s1mod}
\end{align}
consisting of disks with marked points satisfying some additional constraints. The maps $\mathcal{OC}_m$ are defined by counting pearly trees with a single vertex corresponding to a disk with domain in~\eqref{ocmod}, while the moduli spaces~\eqref{s1mod} are used in a similar way to define auxiliary operations arising in the proof of Lemma~\ref{equivariance}.

The elements of~\eqref{ocmod} are disks with $k$ cyclically ordered boundary marked points denoted $z_1,\ldots,z_k$ and $\ell+m+1$ interior marked points denoted
\[ w_1,\ldots,w_{\ell},p_{\out},p_1,\ldots,p_m.\]
The last $m$ of these marked points are called auxiliary, and $p_{\out}$ is called the output marked point. Additionally, on the unit disk representative of such a disk which takes $z_k$ to $1$ and $p_{\out}$ to $0$, the norms of the points $p_i$ are required to satisfy
\begin{align}
0<|p_1|<\cdots<|p_m|<\frac{1}{2}. \label{ocmodnorms}
\end{align}
Define $\theta_i\coloneqq\arg(p_i)$ to be the argument of $p_i$ taken with respect to the unit disk representative.
\begin{figure}
\begin{tikzpicture}
\draw[fill = gray!40] (0,0) circle (4 cm);
\draw[dashed] (0,0) circle (1 cm);
\draw[dashed] (0,0) circle (2 cm);
\draw[dashed] (0,0) circle (3 cm);

\node at (0:2.5) {$\cdots$};

\coordinate (z1) at (-15:4);
\draw[fill = black] (z1) circle (2pt);

\coordinate (z2) at (72:4);
\draw[fill = black] (z2) circle (2pt);

\coordinate (z3) at (144:4);
\draw[fill = black] (z3) circle (2pt);

\coordinate (z4) at (216:4);
\draw[fill = black] (z4) circle (2pt);

\coordinate (zk) at (270:4);
\node[anchor = north] at (zk) {$z_k$};
\draw[fill = black] (zk) circle (2pt);

\coordinate (pout) at (0,0);
\node[anchor = north] at (pout) {$p_{\out}$};
\draw[fill = white] (pout) circle (2pt);

\coordinate (p1) at (11:1cm);
\draw[fill = white] (p1) circle (2pt);
\node[anchor = west] at (p1) {$p_1$};

\coordinate (p2) at (109: 2cm);
\draw[fill = white] (p2) circle (2pt);
\node[anchor = south] at (p2) {$p_2$};

\coordinate (pm) at (192: 3 cm);
\draw[fill = white] (pm) circle (2pt);
\node[anchor = west] at (pm) {$p_m$};

\coordinate (w1) at (150:1.5 cm);
\draw[fill = black] (w1) circle (2pt);

\coordinate (w2) at (280:1.7 cm);
\draw[fill = black] (w2) circle (2pt);

\coordinate (wl) at (55: 2.5 cm);
\draw[fill = black] (wl) circle (2pt);
\node[anchor=east] at (wl) {$w_{\ell}$};
\end{tikzpicture}\caption{An element of~\eqref{ocmod}. The black interior marked points are the ones which are neither auxiliary nor the output.}
\end{figure}
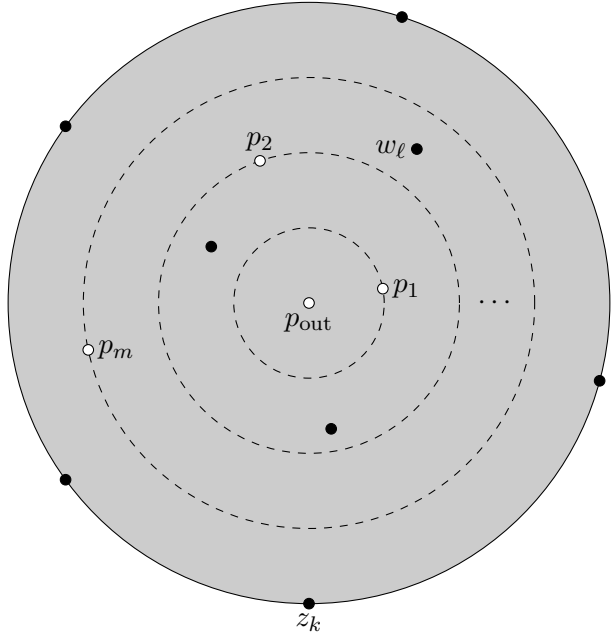

Elements of~\eqref{s1mod} are disks with $k$ cyclically ordered boundary marked points $z_1,\ldots,z_k$ and $\ell+m+2$ interior marked points denoted
\[ w_1,\ldots,w_{\ell},p_{\out},p_1,\ldots,p_{m+1} \]
where the last $m+1$ of these marked points are auxiliary. On the unit disk representative taking $z_k$ to $1$ and $p_{\out}$ to $0$, the norms of the auxiliary marked points are required to satisfy
\[ 0 <|p_1|<\cdots<|p_m|<|p_{m+1}| = \frac{1}{2}. \]
We have an abstract identification $\mathcal{P}^{S^1}_{m,k,\ell}\cong S^1\times\mathcal{P}_{m,k,\ell}$ where the $S^1$-coordinate is given by $\theta_{m+1}$. Under this identification, we orient~\eqref{s1mod} by giving $S^1$ the opposite of its boundary orientation, and giving the product the opposite of the product orientation.

We form uncompactified moduli spaces of pseudoholomorphic disks in $M$ with boundary on $L$ whose domains belong to the moduli spaces~\eqref{ocmod} or~\eqref{s1mod}. These are denoted
\begin{align}
\mathcal{P}_{m,k,\ell}(\beta) \label{ocmodbeta}\\
\mathcal{P}_{m,k,\ell}^{S^1}(\beta) \label{s1modbeta}
\end{align}
Each of these moduli spaces has naturally defined evaluation maps at each of the boundary and interior marked points.

To define the pearly trees relevant to the open-closed map, we need to modify our orientation convention for trees. 
\begin{definition}
Let $T$ be a bicolored tree in the sense of Definition~\eqref{bicoloredtree} equipped with a ribbon structure, a metric, and a labeling of its vertices by classes $\beta_v\in H_2(M,L)$ for $v\in E_{\circ}(T)$ and $\beta_v\in H_2(M)$ for $v\in E_{\bullet}(T)$. Suppose that the edge set of $T$ can be written as
\begin{align*}
E_{\circ}(T) &= \lbrace e_1^{\circ},\ldots,e_k^{\circ}\rbrace \\
E_{\bullet}(T) &= \lbrace e_1^{\bullet},\ldots,e_{\ell}^{\bullet},e_{\out}^{\bullet}\rbrace
\end{align*}
We say that $T$ is of \textit{open-closed type} if its orientation is obtained by declaring that
\begin{itemize}
\item[•] $e^{\bullet}_{\out}$ is an outgoing edge adjacent to the vertex $v_{\out}\coloneqq s(e^{\bullet}_{\out})\in E_{\circ}(T)$, and all other edges of $T$ point toward $v_{\out}$.
\end{itemize}
\end{definition}
The choices of homology classes and the valences of each vertex determine associated moduli spaces $\mathcal{M}(\beta_v)\coloneqq\mathcal{M}_{\val_{\circ}(v)-1,\val_{\bullet}(v)}(\beta_v)$ if $v\in E_{\circ}(T)$ or $\mathcal{M}(\beta_v)\coloneqq\mathcal{M}_{\val_{\bullet}(v)}(\beta_v)$ if $v\in E_{\bullet}(T)$. To the vertex $v_{\out}$, we associate one of the moduli spaces \begin{align}
\mathcal{P}_{m,\val_{\circ}(v_{\out}),\val_{\bullet}(v_{\out})}(\beta_{v_{\out}}) \\
\mathcal{P}^{S^1}_{m,\val_{\circ}(v_{\out}),\val_{\bullet}(v_{\out})+1}(\beta_{v_{\out}}).
\end{align}
In particular, auxiliary marked points do not have corresponding edges of $T$, but all other interior and boundary marked points do. The bicoloring on $T$ induces an evaluation map $\ev_T$ defined similarly to~\eqref{evtree}. Finally, for each edge $e\in E_{\circ}(T)$, choose an $f_L$-admissible Morse function on $L$, and for each edge $e\in E_{\bullet}(T)$, choose an $f_M$-admissible Morse function on $M$. For each combinatorially finite edge $e\in E^{f}_{\circ/\bullet}(T)$, we can define embeddings as in~\eqref{graphcirc} and~\eqref{graphbullet}, the images of which are still denoted $G_e$. Given the data as above, we can now define the moduli spaces of pearly trees contributing to the cyclic open-closed map.
\begin{definition}\label{ocpearlspaces}
Let $x = (x_1,\ldots,x_k)$ be a sequence of input critical points in $\crit(f_L)$, let $y = (y_1,\ldots,y_{\ell})$ be a sequence of input critical points in $\crit(f_M)$, and let $y_{\out}\in\crit(f_M)$ be an output critical point. The \textit{moduli spaces of open-closed pearly trees} of class $\beta\in H_2(M,L)$ are denoted
\begin{align*}
\mathcal{P}_{m}(x;y,y_{\out};\beta)\\
\mathcal{P}_{m}^{S^1}(x;y,y_{\out};\beta) 
\end{align*}
and are defined to be
\begin{align*}
\coprod_T\ev_T^{-1}\left(\prod_{j=1}^{\ell}W^u(y_j)\times\prod_{j=1}^k W^u(x_j)\times\prod_{e\in E^f_{\circ}(T)\sqcup E^f_{\bullet}(T)}G_e\times W^s(y_{\out})\right)
\end{align*}
where the output vertex $v_{\out}$ is associated an element of~\eqref{ocmodbeta} or~\eqref{s1modbeta}, respectively. Both of these spaces admit natural Gromov compactifications
\begin{align}
\overline{\mathcal{P}}_{m}(x;y,y_{\out};\beta) \label{ocpearls}\\
\overline{\mathcal{P}}_{m,k,\ell}^{S^1}(x;y,y_{\out};\beta). \label{s1pearls}
\end{align}
\end{definition}
\begin{figure}
\begin{tikzpicture}
\begin{scope}[scale = 0.75]
\node (N) at (0,0) {};
\node (E) at (0,4) {$y_{\out}$};

\begin{scope}[very thick,decoration={
	markings,
	mark=at position 0.75 with {\arrow{>}}}]
\draw[postaction={decorate}, thick] (N) -- (E);
\end{scope}

\begin{scope}[very thick,decoration={
	markings,
	mark=at position 0.5 with {\arrow{>}}}]
\draw[postaction = {decorate}, thick] (3,4) -- (0,0);
\draw[postaction={decorate}, thick] (3,7) -- (3,4);
\end{scope}
\node[anchor = south] at (3,7) {$y_2$};

\fill[white] (0,0) circle (2cm);
\fill[white] (3,4) circle (1cm);

\node[anchor = south] at (6,4) {$y_1$};
\begin{scope}[very thick,decoration={
	markings,
	mark=at position 0.25 with {\arrow{>}}}]
\draw[postaction={decorate}, thick] (6,4) -- (6,0);
\end{scope}
\fill[white] (6,0) circle (2cm);

\draw (2,0) arc (0:180:2);
\draw (-2,0) arc (180:360:2 and 0.6);
\draw[dashed] (2,0) arc (0:180:2 and 0.6);

\draw (8,0) arc (0:180:2);
\draw (4,0) arc (180:360:2 and 0.6);
\draw[dashed] (8,0) arc (0:180:2 and 0.6);

\coordinate (A) at (2,0);
\coordinate (B) at (4,0);

\begin{scope}[very thick,decoration={
	markings,
	mark=at position 0.5 with {\arrow{>}}}]
\draw[postaction={decorate}, thick] (B) -- (A);
\draw[postaction={decorate}, thick] (10,0) -- (8,0);
\draw[postaction={decorate}, thick] (-4,0) -- (-2,0);
\end{scope}

\node[anchor = west] at (10,0) {$x_1$};
\node[anchor = east] at (-4,0) {$x_2$};

\begin{scope}[xshift = 3cm, yshift = 4cm]
\draw (0,0) circle (1cm);
\draw (-1,0) arc (180:360:1 and 0.3);
\draw[dashed] (1,0) arc (0:180:1 and 0.3);
\end{scope}

\node at (0,-1) {$\mathcal{P}_{0,\val_{\circ}(v_{\out}),\val_{\bullet}(v_{\out})+1}(\beta_{v_{\out}})$};
\node at (6,-1) {$\mathcal{M}_{\val_{\circ}(v)+1,\val_{\bullet}(v)}(\beta_v)$};
\end{scope}
\end{tikzpicture}\caption{A pearly tree in $\mathcal{P}(x_1,x_2;y_{\out},y_1,y_2;\beta)$.}
\end{figure}
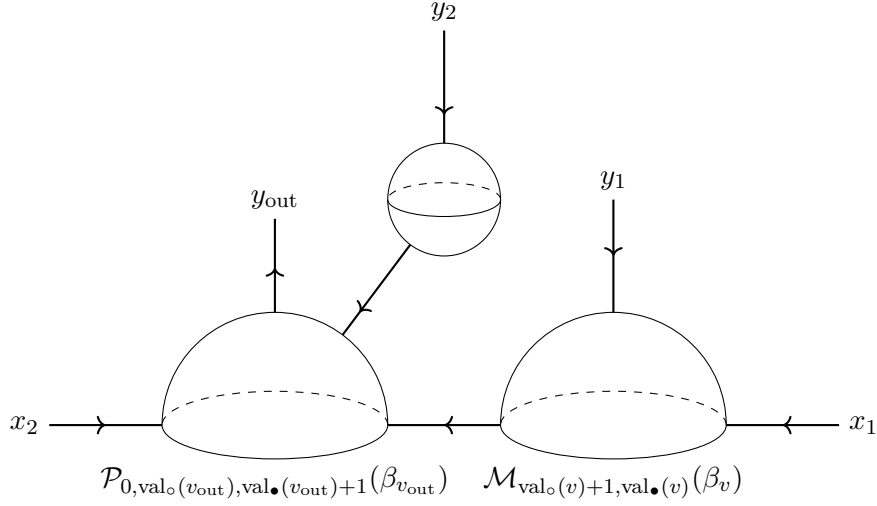
\begin{remark}
To incorporate domain-dependent perturbations in Definition~\eqref{ocpearlspaces}, one would need to consider domain-dependent perturbations on disks satisfying the conditions laid out in~\cite{Gan23}. One would also have to formulate consistency conditions for the perturbation data associated to trees of open-closed type along the lines of those in~\cite{CW22} or~\cite{VWX20}.
\end{remark}
The regularity of~\eqref{ocpearls} and~\eqref{s1pearls} do not immediately follow from Assumption~\ref{regularity}, so we must also assume that:
\begin{assumption}\label{regularityoc}
The moduli spaces~\eqref{ocpearls} and~\eqref{s1pearls} (defined with respect to the almost complex structure $J$ subject to Assumption~\ref{regularity}) of open-closed pearly trees of virtual dimension $\leq 1$ are compact oriented orbifolds of the expected dimension.
\end{assumption}
This assumption is sufficient for us to define the components of the cyclic open-closed map.
\begin{definition}
Given sequences of input critical points $x$ and $y$ as in Definition~\ref{ocpearlspaces}, define
\begin{align*}
\mathfrak{oc}^{\beta}_{m,k,\ell}(x;y)&\coloneqq\sum_{y_0\in\crit(f_M)}(-1)^{\star_k}\hol_{\nabla}(\beta)\#|\mathcal{P}_{m,k,\ell}(x;y,y_{\out};\beta)|y_{\out} \\
\mathfrak{oc}^{S^1,\beta}_{m,k,\ell}(x;y)&\coloneqq\sum_{y_{\out}\in\crit(f_M)}(-1)^{\star^{S^1}_k}\hol_{\nabla}(\beta)\#|\mathcal{P}_{m,k,\ell}^{S^1}(x;y,y_{\out};\beta)|y_{\out}
\end{align*}
and extend these $R$-linearly. Here $\#|\mathcal{P}_{m,k,\ell}(x;y,y_{\out};\beta)|$ and $\#|\mathcal{P}_{m,k,\ell}^{S^1}(x;y,y_{\out};\beta)|$ are signed counts of elements in the respective moduli spaces, and the signs are determined by
\begin{align*}
\star_k = \sum_{j=1}^k (n+j)|\alpha_j|' \\
\star^{S^1}_k = \star_k+\maltese_k-1.
\end{align*}

Define operations
\begin{align*}
\mathfrak{oc}_{m,k,\ell}(x;y)&\coloneqq\sum_{\beta\in H_2(M,L;\mathbb{Z})}\mathfrak{oc}^{\beta}_{m,k,\ell}(x;y) \\
\mathfrak{oc}^{S^1}_{m,k,\ell}(x;y)&\coloneqq\sum_{\beta\in H_2(M,L;\mathbb{Z})}\mathfrak{oc}^{S^1,\beta}_{m,k,\ell}(x;y). \\
\end{align*}
For any bulk deformation parameter $\gamma$, define the bulk-deformed operations
\begin{align}
\mathcal{OC}_{m,k}(x)\coloneqq\sum_{\ell=0}^{\infty}\frac{1}{\ell!}\mathfrak{oc}_{m,k,\ell}(x;\gamma^{\otimes\ell}) \\
\mathcal{OC}_{m,k}^{S^1}(x)\coloneqq\sum_{\ell=0}^{\infty}\frac{1}{\ell!}\mathfrak{oc}_{m,k,\ell}^{S^1}(x;\gamma^{\otimes\ell}).
\end{align}
\end{definition}
We prove Lemma~\ref{equivariancelemma} by examining the boundary strata of the moduli spaces of open-closed pearls. First, we describe the boundary strata of~\ref{ocpearls}.
\begin{lemma}\label{ocpearlsboundary}
Consider a sequence of inputs $y = y_1\otimes\cdots\otimes y_{\ell}\in\ CM^*(M;Q)^{\otimes\ell}$, an input sequence $x = x_1\otimes\cdots\otimes x_k\in CM^*(L;R)^{\otimes k}$, an output $y_{\out}\in CM_*(M;Q)$, and a class $\beta\in H_2(M,L;\mathbb{Z})$ such that~\eqref{ocpearls} is $1$-dimensional. Let $I_1\sqcup I_2 = \lbrace 1,\ldots,\ell\rbrace$ denote a partition of the set of positive integers $\leq\ell$ into disjoint ordered subsets, and let $\beta = \beta_1+\beta_2\in H_2(M,L;\mathbb{Z})$. Then the boundary of~\eqref{ocpearls} is covered by the images under the natural inclusions of the following products of zero-dimensional moduli spaces:
\begin{align}
\mathcal{M}(x_{\out};x_{i+1},\ldots,x_{i+j};y_{I_1},y_{\out};\beta_1)&\times\mathcal{P}_m(x_1,\ldots,x_i,x_{\out},x_{i+j+1},\ldots,x_k;y_{I_2},y_{\out};\beta_2) \label{ocstdbreaks}\\
&\mathcal{P}^{S^1}_{m-1}(x;y,y_{\out};\beta) \label{normsto12}\\
&\mathcal{P}^{i,i+1}_m(x;y,y_{\out};\beta)\label{normscoincide}
\end{align}
where~\eqref{normscoincide} consists of the subset of pearls whose output vertex is decorated by a pseudoholomorphic disk whose domain is of the sort contained in~\eqref{ocmod}, except that for some $1\leq i\leq m-1$, the norms of the auxiliary marked points in the unit disk representative satisfy $|p_i| = |p_{i+1}|$.
\end{lemma}
\begin{proof}
The first type of boundary breaking is of the same sort that occurs in one-dimensional moduli spaces of ordinary pearly trajectories. These behave as expected because we have used a fixed $J$ to define the Cauchy--Riemann equation. Because of the constraints on their norms, the auxiliary marked points must always remain on the same component of any Gromov limit of a sequence such of curves. The boundary components~\eqref{normsto12} and~\eqref{normscoincide} arise from sequences of pearly trajectories coming from a sequence of holomoprhic disks in which the norms of the auxiliary marked points change.
\end{proof}
\begin{lemma}\label{hochs1}
For each $m\geq0$, we have that
\[ \mathcal{OC}^{S^1}_{m-1} + \mathcal{OC}_m\circ b = 0 \, .\]
\end{lemma}
\begin{proof}
Consider a sequence of input critical points $x = (x_1,\ldots,x_k)$ in $\crit(f_L)$ and an output critical point $y_{\out}\in\crit(f_M)$ for which the moduli spaces $\mathcal{P}_{m}(x;\gamma^{\otimes\ell},y_{\out};\beta)$ are one-dimensional, where $\gamma$ is a bulk parameter. Lemma~\ref{ocpearlsboundary} implies that we can write
\[ 0 = \mathcal{OC}^{S^1}_{m-1}+\mathcal{OC}_m\circ b+\sum_{i=1}^{m-1}{}^{i,i+1}\mathcal{OC}_{m}\]
where the operations ${}^{i,i+1}\mathcal{OC}_m$ are defined by counting pearly trees in moduli spaces of the form~\eqref{normscoincide}. The first summand corresponds to the boundary components~\eqref{normsto12}, and the second summand corresponds to the boundary components~\eqref{ocstdbreaks}.

All of the operations in the last summand vanish, because there is a forgetful map
\begin{align}
\mathcal{P}_m^{i,i+1}(x,y,y_{\out};\beta)\to\mathcal{P}_{m-1}(x,y,y_{\out};\beta) \label{forgetful}
\end{align}
which forgets the auxiliary marked point $p_{i+1}$, and shifts the labels of all remaining auxiliary marked points down by $1$. This implies that the elements of $\mathcal{P}_m^{i,i+1}(x,y,y_{\out};\beta)$ are never isolated, as this forgetful map always has one-dimensional fibers. The lemma now follows from a sign analysis of the sort carried out in~\cite{Gan23}.
\end{proof}
\begin{remark}
For us, the existence of the forgetful map~\eqref{forgetful} follows because we have defined all moduli spaces of pseudoholomorphic disks using a fixed almost complex structure. In~\cite{Gan23}, the conditions imposed on domain-dependent perturbations for open-closed moduli spaces imply that the relevant analogue~\eqref{forgetful} exists. By~\cite[Remark 3.1]{F10}, we cannot expect to construct Kuranishi structures which are compatible with forgetful maps of interior marked points. Nevertheless, one might hope for an independent proof in that setting that there are no isolated pearly trees in $\mathcal{P}_m^{i,i+1}(x,y,y_{\out};\beta)$.
\end{remark}
In~\cite[Proposition 12]{Gan23}, Deligne--Mumford type compactifications of~\eqref{s1mod} are decomposed into sectors corresponding to the angle of the auxiliary marked point $p_1$. The following lemma is proved by showing, roughly, that this decomposition into sectors yields a decomposition of the moduli spaces~\eqref{s1pearls} of pearly trees.
\begin{lemma}\label{sectors}
For each $m\geq0$, we have that
\[ \mathcal{OC}^{S^1}_m = \mathcal{OC}_m\circ B \, .\]
\end{lemma}
\begin{proof}
Let
\[ \mathcal{P}_{m,k+1,\ell,\tau_i}\]
to be the uncompactified moduli space of disks with $k+1$ cyclically ordered boundary marked points
\[z_1,\ldots,z_i,z_{0},z_{i+1},\ldots,z_k\]
along with $\ell$ interior marked points, $m$ auxiliary interior marked points, and an output interior marked point $p_{\out}$. The norms (on the unit disk representative) of the auxiliary marked points are required to satisfy
\[ 0<|p_1|<\cdots<|p_m|<\frac{1}{2}.\]
There is a bijection
\[ \tau_i\colon\mathcal{P}_{m,k+1,\ell,\tau_i}\to\mathcal{P}_{m,k,\ell}\]
given by cyclically permuting boundary labels. 

Now consider, for all $1\leq i \leq k$, the moduli spaces
\[ \mathcal{P}_{m,k,\ell}^{S^1_{i,i+1}} \]
which are the open subsets of $\mathcal{P}_{m,k,\ell}^{S^1}$ with the property that $\arg(p_1)$ lies between $\arg(z_i)$ and $\arg(z_{i+1})$, where the indices are taken mod $k$.

We also have an auxiliary-rescaling map
\[ \pi_{k+1}^i\colon\mathcal{P}_{m,k+1,\ell,\tau_i}\to\mathcal{P}^{S^1_{i,i+1}}_{m,k,\ell}\]
which places a marked point $p_{m+1}$ of norm $\frac{1}{2}$  with the line between $p_{\out}$ and $z_{k+1}$, and deletes $z_{k+1}$. Taking the union of these maps gives an orientation-preserving embedding
\[\coprod_i\mathcal{P}_{m,k+1,\ell,\tau_i}\xrightarrow{\coprod_i\pi^i_{k+1}}\coprod_i\mathcal{P}_{m,k,\ell}^{S^1_{i,i+1}}\hookrightarrow\mathcal{P}_{m,k,\ell}^{S^1}.\]
It is clear that image of this embedding covers all but a codimension $1$ subset of the target. Specifically, the complement of the image is the locus of disks for which $\arg(p_1) = \arg(z_i)$ for some $1\leq i\leq k$. 

This implies that all elements of the zero-dimensional moduli spaces~\eqref{ocpearls} can be taken to have disks at the output vertices whose domains lie in the image of this embedding, possibly after perturbing the Morse functions $f_L$ and $f_M$. By counting isolated pseudoholomorphic pearly trees with underlying bicolored metric ribbon trees of open-closed type and output vertex decorated by elements of moduli spaces of the form
\[ \mathcal{P}_{m,k+1,\ell,\tau_i}(\beta)\]
we can define operations $\mathcal{OC}_{m,k,\tau_i}$. More precisely, elements of this moduli space consist of pseudoholomorphic disks $(D^2,\partial D^2)\to(M,L)$ representing the class $\beta\in H_2(M,L;\mathbb{Z})$ whose domains lie in $\mathcal{P}_{m,k+1,\ell,\tau_i}$.

Hence, if we assume without loss of generality that $f_L$ has a unique minimum which represents the unit $1\in CM^*(L;R)$, there is an equality of chain-level operations
\begin{align*}
\mathcal{OC}^{S^1}_{m,k}(x_1\otimes\cdots\otimes x_k) &= \sum_{i=0}^{m-1}\mathcal{OC}_{m,k,\tau_i}(x_1\otimes \cdots\otimes x_i\otimes 1\otimes x_{i+1}\otimes\cdots\otimes x_k) \\
&= \sum_{i=0}^{m-1}\mathcal{OC}_{m,k+1}(1\otimes x_{i+1}\otimes\cdots\otimes x_k\otimes x_1\otimes\cdots\otimes x_i).
\end{align*}
In this identity, the inputs for $\mathcal{OC}_{m,k,\tau_i}$ at the marked point $z_0$ must be $1$ for degree reasons. The last equality holds because the bijections $\tau_i$ induce bijections between the relevant spaces of pearly trees. The result now follows from a sign analysis of the operations on the right hand side, of the sort carried out in~\cite{Gan23}.
\end{proof}
\begin{proof}[Proof of Lemma~\ref{equivariancelemma}]
This is an immediate consequence of Lemma~\ref{hochs1} and Lemma~\ref{sectors}.
\end{proof}
There is also a version of the cyclic open-closed map on the possibly bulk-deformed $A_{\infty}$-algebras
\[\widetilde{\mathcal{A}} = (CM^*(L\times[-1,1];R),\lbrace\widetilde{\mathfrak{m}}^{\widetilde{\gamma}}_k\rbrace_{k=0}^{\infty})\]
as defined in Section~\ref{cylinderobjects} using a path of almost complex structures $\underline{J} = \lbrace J_t\rbrace_{t\in[-1,1]}$. The definitions of these maps require time-dependent analogues of~\eqref{ocmodbeta} and~\eqref{s1modbeta}. Namely consider the moduli spaces
\begin{align}
\widetilde{\mathcal{P}}_{m,k,\ell}(\beta)\coloneqq\lbrace(u,t)\colon u\in\mathcal{P}_{m,k,\ell}(\beta;J_t) \label{ocmodbetacylinder} \\
\widetilde{\mathcal{P}}_{m,k,\ell}^{S^1}(\beta)\coloneqq\lbrace(u,t)\colon u\in\mathcal{P}^{S^1}_{m,k,\ell}(\beta;J_t). \label{s1modbetacylinder}
\end{align}
Given a tree $T$ of open-closed type, we can define open-closed pearly trees on cylinder objects.
\begin{definition}\label{ocpearlspacescylinder}
Let $\widetilde{x} = (\widetilde{x}_1,\ldots,{x}_k)$ be a sequence of input critical points in $\crit(\widetilde{f}_L)$, let $y = (\widetilde{y}_1,\ldots,\widetilde{y}_{\ell})$ be a sequence of input critical points in $\crit(\widetilde{f}_M)$, and let $\widetilde{y}_{\out}\in\crit(\widetilde{f}_M)$ be an output critical point. The \textit{time-dependent moduli spaces of open-closed pearly trees} of class $\beta\in H_2(M,L)$ are denoted
\begin{align*}
{\mathcal{P}}_{m}(\widetilde{x};\widetilde{y},\widetilde{y}_{\out};\beta)\\
{\mathcal{P}}_{m}^{S^1}(\widetilde{x};\widetilde{y},\widetilde{y}_{\out};\beta) 
\end{align*}
and are defined to be
\begin{align*}
\coprod_T\ev_T^{-1}\left(\prod_{j=1}^{\ell}W^u(y_j)\times\prod_{j=1}^k W^u(x_j)\times\prod_{e\in E^f_{\circ}(T)\sqcup E^f_{\bullet}(T)}G_e\times W^s(y_{\out})\right)
\end{align*}
where the output vertex $v_{\out}$ is associated an element of~\eqref{ocmodbetacylinder} or~\eqref{s1modbetacylinder}, respectively. Both of these spaces admit natural Gromov compactifications
\begin{align}
\overline{{\mathcal{P}}}_{m}(x;y,y_{\out};\beta) \label{ocpearlscylinder}\\
\overline{{\mathcal{P}}}_{m}^{S^1}(x;y,y_{\out};\beta). \label{s1pearlscylinder}
\end{align}
\end{definition}
As usual, we work under a regularity assumption on these time-dependent pearly moduli spaces.
\begin{assumption}\label{occylinderregularity}
The moduli spaces~\eqref{ocpearlscylinder} and~\eqref{s1pearlscylinder} of virtual dimension $\leq 1$ are compact orbifolds of the expected dimension.
\end{assumption}
Given this assumption, a discussion completely parallel to the one carried out above shows that we can construct the desired cyclic open-closed map on the cylinder.
\begin{lemma}
There exists a sequence of linear maps
\[ \widetilde{\mathcal{OC}}_m\colon CH_*(\widetilde{\mathcal{A}})\to CM_{n+1-*+2m}(M\times[-1,1]) \]
which satisfy
\[ \widetilde{\mathcal{OC}}_{m-1}\circ B+\widetilde{\mathcal{OC}}_m\circ b = 0\]
for all $m\geq0$. \qed
\end{lemma}

\section{The open Gromov--Witten potential}\label{theogw} In this section, we explain how to construct the open Gromov--Witten potential using the $\infty$-inner product induced by the cyclic open-closed map. We have so far not discussed disks without marked points. Since we have defined the Lagrangian Floer theory using pearly configurations, the appropriate analogue of the inhomogeneous $\mathfrak{m}_{-1}$ as it appears in~\cite{F11} and~\cite{ST21} should of course be a count of pearly trees, not just of disks.
\subsection{Inhomogeneous terms}
Denote by $\mathcal{M}_{-1,\ell}(\beta;J)$ the moduli spaces of pseudoholomorphic disks $u\colon(D^2,\partial D^2)\to(M,L)$ in the class $\beta\in H_2(M,L;\mathbb{Z})$ with no boundary marked points and interior marked points labeled
\[ w_1,\ldots,w_{\ell}\]
in order for any integer $\ell\geq0$. There are evaluation maps \begin{align*}
\evi_j\colon\mathcal{M}_{-1,\ell}(\beta;J)\to M
\end{align*}
at the interior marked points. The pearly trees which contribute to the inhomogeneous term of our open Gromov--Witten potential should be trees with no outputs which take no inputs from the Morse cochain of $L$.
\begin{definition}
Let $T$ be a bicolored tree equipped with a metric, a ribbon structure, and a labeling of its vertices by by classes $\beta_v\in H_2(M,L)$ for $v\in E_{\circ}(T)$ and $\beta_v\in H_2(M)$ for $v\in E_{\bullet}(T)$. We say that $T$ is of \textit{inhomogeneous type} if $E_{\circ}(T)$ contains no semi-infinite edges, and if its oriented such that
\begin{itemize}
\item[•] all semi-infinite edges in $E_{\bullet}(T)$ are are incoming edges; and
\item[•] every disk vertex $v\in E_{\circ}(T)$ has at most one outgoing adjacent edge.
\end{itemize}
\end{definition}

\begin{figure}
\begin{tikzpicture}
\begin{scope}[scale = 0.75]
\node (N) at (0,0) {};
\node (E) at (0,4) {$y_{0}$};

\begin{scope}[very thick,decoration={
	markings,
	mark=at position 0.25 with {\arrow{>}}}]
\draw[postaction={decorate}, thick] (E) -- (N);
\end{scope}

\begin{scope}[very thick,decoration={
	markings,
	mark=at position 0.5 with {\arrow{>}}}]
\draw[postaction = {decorate}, thick] (3,4) -- (0,0);
\end{scope}

\fill[white] (0,0) circle (2cm);
\node[anchor = south] at (3,4) {$y_2$};
\node[anchor = south] at (6,4) {$y_1$};
\begin{scope}[very thick,decoration={
	markings,
	mark=at position 0.25 with {\arrow{>}}}]
\draw[postaction={decorate}, thick] (6,4) -- (6,0);
\end{scope}
\fill[white] (6,0) circle (2cm);

\draw (2,0) arc (0:180:2);
\draw (-2,0) arc (180:360:2 and 0.6);
\draw[dashed] (2,0) arc (0:180:2 and 0.6);

\draw (8,0) arc (0:180:2);
\draw (4,0) arc (180:360:2 and 0.6);
\draw[dashed] (8,0) arc (0:180:2 and 0.6);

\coordinate (A) at (2,0);
\coordinate (B) at (4,0);

\begin{scope}[very thick,decoration={
	markings,
	mark=at position 0.5 with {\arrow{>}}}]
\draw[postaction={decorate}, thick] (B) -- (A);
\end{scope}
\end{scope}
\end{tikzpicture}\caption{A pearly tree contributing to $\mathfrak{m}_{-1}$.}
\end{figure}
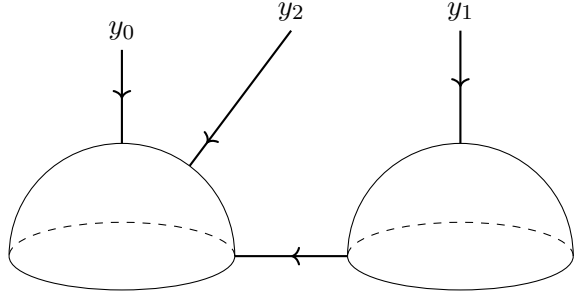
For each $T$ of inhomogeneous type, choose $f_M$-admissible and $f_L$-admissible Morse functions on each edge $E_{\bullet}(T)$ and $E_{\circ}(T)$, respectively.

\begin{definition}\label{pearlytreesm-1def}
Given input critical points $y_1,\ldots,y_{\ell}\in\crit(f_M)$, define
\begin{align}
\mathcal{M}_{-1}(y_1\otimes\cdots\otimes y_{\ell};\beta)\label{pearlytreesm-1}
\end{align}
to be the moduli space of all pearly trees with the given inputs whose domain is parametrized by a tree $T$ of inhomogeneous type.
\end{definition}
The precise definition of these moduli spaces also uses an evaluation map $\ev_T$ associated to $T$, but the submanifold that we pull back under this map does not contain any stable manifold factor (which would correspond to an output). We need a regularity assumption on these moduli spaces which is the analogue of Assumption~\ref{regularity}.
\begin{assumption}\label{m-1regularity}
The moduli spaces~\eqref{pearlytreesm-1} of virtual dimension $0$ are compact oriented zero-dimensional orbifolds.
\end{assumption}
This assumption lets us define operations by counting the elements of the zero-dimensional moduli spaces of this type.
\begin{definition}
Let $y_1,\ldots,y_{\ell}\in\crit(f_M)$ be a sequence of input critical points and $y_{\out}$ be a critical point of index $0$, with the property that~\eqref{pearlytreesm-1} is zero-dimensional, and set
\begin{align*}
\mathfrak{q}_{-1,\ell}^{\beta}(y_1\otimes\cdots\otimes y_{\ell}) = \sum_{y_{\out}\in\crit(f_M)}\hol_{\nabla}(\beta)\#|\mathcal{M}_{-1}(y_1\otimes\cdots\otimes y_{\ell};\beta)|\in R.
\end{align*}
Extend these to operations
\begin{align*}
\mathfrak{q}_{-1,\ell}\colon CM^*(M;Q)^{\otimes\ell}\to R
\end{align*}
by setting
\begin{align*}
\mathfrak{q}_{-1,\ell}\coloneqq\sum_{\beta\in H_2(M,L;\mathbb{Z})}\mathfrak{q}_{-1,\ell}^{\beta}.
\end{align*}
Finally, given a bulk parameter $\gamma\in CM^*(M)$, set
\begin{align}
\mathfrak{m}_{-1}^{\gamma}\coloneqq\sum_{\ell=0}^{\infty}\frac{1}{\ell!}\mathfrak{q}_{-1,\ell}(\gamma^{\otimes\ell}).
\end{align}
\end{definition}
Given a path $\underline{J} = \lbrace J_t\rbrace_{t\in[-1,1]}$ of almost complex structures, there are also moduli spaces
\[ \widetilde{\mathcal{M}}_{-1,\ell}(\beta;\underline{J})\coloneqq\lbrace(u,t)\colon u\in\mathcal{M}_{-1,\ell}(\beta;J_t)\rbrace \]
which have naturally defined evaluation maps at the interior marked points. Using these evaluation maps, one defines moduli spaces of pearly trees
\begin{align}
\mathcal{M}_{-1}(\widetilde{y}_1\otimes\cdots\otimes\widetilde{y}_{\ell};\beta;\underline{J})\label{pearlytreescylm-1}
\end{align}
with underlying tree $T$ of inhomogeneous type.
\begin{assumption}\label{m-1cylregularity}
The moduli spaces~\eqref{pearlytreescylm-1} of virtual dimension at most $1$ are compact orbifolds of the expected dimension.
\end{assumption}
We need regularity for $1$-dimensional time-dependent moduli spaces, since these arise in the proof of Theorem~\ref{truemain}, whereas the $1$-dimensional moduli spaces of ordinary pearly trees with no inputs do not. With this we define the time-dependent inhomogeneous terms $\widetilde{\mathfrak{m}}_{-1}^{\widetilde{\gamma}}$ in the obvious way.

If $T$ is a tree for which the planar part $T_{\circ}$ consists of a single vertex, the boundary of an associated pearly configuration can collapse to a point in $L$. We will now introduce notation for the count of such configurations.

For $\beta\in H_2(M;\mathbb{Z})$, define the moduli spaces of $\underline{J}$-holomorphic spheres
\[\widetilde{\mathcal{M}}_{\ell+1}(\beta;\underline{J})\coloneqq\lbrace(u,t)\colon u\in\mathcal{M}_{\ell+1}(\beta;J_t)\rbrace\]
and label the marked points on the domain $0,\ldots,\ell$. Consider bicolored metric ribbon trees $T$ for which $T_{\circ} = \emptyset$ with $\ell+1$ semi-infinite edges, labeled $e_{\out}^{\bullet},e_1^{\bullet},\ldots,e_{\ell}^{\bullet}$, where $e_{\out}^{\bullet}$ is outgoing all other semi-infinite edges are incoming. The combinatorially finite edges of $T$ are oriented so that they point to $e_0^{\bullet}$. Associate a class $\beta_v\in H_2(M;\mathbb{Z})$ to each vertex $v\in V(T) = V_{\bullet}(T)$. Given a sequence of input critical points $\widetilde{y}_1,\ldots,\widetilde{y}_{\ell}\in\crit(\widetilde{f}_M)$, we can define moduli spaces of pearly trees
\begin{equation}
\mathcal{M}(y_{\out},y_1,\ldots,y_{\ell};\beta;\underline{J})\label{GWmod}
\end{equation}
by assigning $\widetilde{f}_M$-admissible Morse functions to each edge of $T$ and using the evaluation maps associated to $T$. As is routine by now, we impose a regularity assumption on these moduli spaces.
\begin{assumption}\label{cylsphereregularity}
The moduli spaces~\eqref{GWmod} of virtual dimension $0$ are compact orbifolds of dimension $0$.
\end{assumption}
We set
\[ \widetilde{\mathfrak{q}}^{\beta}_{\emptyset,\ell}(y_1,\ldots,y_{\ell})\coloneqq\sum_{y_{\out}}\#|\mathcal{M}(y_{\out},y_1,\ldots,y_{\ell};\beta;\underline{J})|\cdot y_{\out}\]
where the sum is over all $y_{\out}$ such that~\eqref{GWmod} is $0$-dimensional. Extend these to operations
\begin{align*}
\widetilde{\mathfrak{q}}_{\emptyset,\ell}\colon CM^*(M\times[-1,1];Q)^{\otimes\ell}\to CM^*(M\times[-1,1];Q)
\end{align*}
in the usual way. The Lagrangian embedding $\iota\colon L\hookrightarrow M$ lets us pull back the values of $\widetilde{\mathfrak{q}}_{\emptyset,\ell}$ to $CM^*(L\times[-1,1])$, the result of which is denoted
\begin{equation}
\iota^*\widetilde{\mathfrak{q}}_{\emptyset,\ell}(\widetilde{y}_1\otimes\cdots\otimes\widetilde{y}_{\ell}). \label{pullbackspherecount}
\end{equation}
By the construction of the Morse function $\widetilde{f}_L$ used to construct $CM^*(L\times[-1,1])$, we can assume that $CM^*(L\times[-1,1];R)$ has a single generator in degree $n+1$. This is because we can take the Morse function $f_0$ on $L$ to have a single maximum without loss of generality. The element $\widetilde{GW}\in R$ is defined to be the coefficient of this degree $n+1$ generator in the Morse cochain~\eqref{pullbackspherecount}.

\subsection{Wall-crossing}
Having defined the inhomogeneous terms $\mathfrak{m}_{-1}^{\gamma}$ in our setting, we can now define the open Gromov--Witten potential as it is defined in the cyclic case. The higher order terms of the open Gromov--Witten potential will be as in Definition~\ref{inftycyclicpotential}. The $\infty$-inner product we use is induced from the trace map
\[ CC_{*}^{+,\red}\xrightarrow{\mathcal{OC}} CM_{*-n}(M;R)\otimes_R R((u))/u R[[u]]\to R\]
where the last map projects to the $u^0$-factor and then projects to $R = H_*(\pt;R)$. This can be thought of as a positive cyclic cocycle, which induces a negative cyclic cohomology class by Lemma~\ref{hciso}, and in turn an $\infty$-inner product by Lemma~\ref{cocycletohom}.
\begin{definition}\label{infinityogwpotential}
For a Lagrangian submanifold $L\subset M$ with a flat $GL(1,\Bbbk)$-connection $\nabla$ and a bulk parameter $\gamma$, let $\mathcal{A} = (CF^*(L;R),\lbrace\mathfrak{m}^{\gamma}_k\rbrace_{k=0}^{\infty})$ be the bulk-deformed pearly $A_{\infty}$-algebra, and let $\phi\colon\mathcal{A}_{\Delta}\to\mathcal{A}^{\vee}$ denote the $\infty$-inner product induced by the cyclic open-closed map. The \textit{$\infty$-open Gromov--Witten potential} of $(L,\nabla)$ is the function $\Phi\colon\mathcal{MC}(\mathcal{A})\to H_0(M;R)$ defined by the convergent power series
\begin{align}
\Phi(b)\coloneqq\mathfrak{m}_{-1}^{\gamma}+\sum_{N=0}^{\infty}\sum_{p+q+k = N}\frac{1}{N+1}\phi_{p,q}(b^{\otimes p}\otimes\underline{\mathfrak{m}^{\gamma}_k(b^{\otimes k})}\otimes b^{\otimes q})(b).
\end{align}
\end{definition}
To show that this choice of inhomogeneous term is appropriate, we need to verify that $\Phi(b)$ has the expected behavior under variations of the almost complex structure $J$ used to construct the open-closed map and the $A_{\infty}$-operations. For this purpose, we need a notion of gauge-equivalence of weak bounding cochains which is compatible with the $A_{\infty}$-structures we have constructed on cylinder objects, adapted from~\cite[Definition 3.12]{ST21}.
\begin{definition}\label{gauge-equivalence}
Suppose that we have a bulk parameter $\widetilde{\gamma}\in\widetilde{\mathcal{A}} = CM^*(M\times[-1,1];Q)$, where curved the $A_{\infty}$-structure is defined using $\underline{J}$, and a class $\widetilde{b}\in\widetilde{\mathcal{A}}$ of degree $|b| = 1$. Further assume that there is a constant $c\in R$ such that
\begin{align*}
\widetilde{\mathfrak{m}}_0^b = c\cdot 1.
\end{align*}
Then the pairs $(b_{-1},\gamma_{-1})$ and $(b_1,\gamma_1)$ obtained by restricting to $L\times\lbrace\mp 1\rbrace$ and $M\times\lbrace\mp 1\rbrace$ are said to be gauge-equivalent.
\end{definition}
The rest of this section is occupied by the proof of our main result.
\begin{theorem}[Wall crossing formula]\label{truemain}
Let $\underline{J} = \lbrace J_t\rbrace_{t\in [-1,1]}$ be a path of almost complex structures satisfying Assumption~\ref{cylregularity}, and suppose we are given gauge-equivalent pairs $(b_{\mp 1},\gamma_{\mp 1})$ which are gauge-equivalent in the sense of Definition~\ref{gauge-equivalence}. Then the open Gromov--Witten potentials defined with respect to the almost complex structures $J_{-1}$ and $J_1$ and Morse functions $f^{-1}_L$ and $f^{1}_L$ satisfy
\[ \Phi_{-1}(b_{-1}) = \Phi_1(b_1)+\widetilde{GW}. \]
\end{theorem}
\begin{proof}
Let $\phi^{\pm 1}$ denote the $\infty$-inner products on $CM^*(L;R;f_L^{\pm 1})$ constructed from the cyclic open-closed maps $\mathcal{OC}^{\pm 1}$ defined using the almost complex structures $J_{\pm 1}$, and let $\widetilde{\phi}$ denote the $\infty$-inner product on $CM^*(L\times[-1,1];R)$ defined using $\underline{J}$. Additionally, let $\widetilde{b}$ and $\widetilde{\gamma}$ be a bounding cochain and bulk parameter on $CM^*(L\times[-1,1];R)$ realizing the gauge-equivalence between $b_{-1}$ and $b_1$.

Let $\pi\colon M\times[-1,1]\to[-1,1]$ denote the projection map, and let 
\[\pi_*\colon(CM_*(M\times[-1,1]),\partial)\to(CM_*([-1,1]),\partial)\]
denote the induced map on Morse cochains, where the Morse compex on $[-1,1]$ is defined using the Morse function $\rho$ on $(-1-\epsilon,1+\epsilon)$. Since this is a chain map, it follows that
\begin{align}\label{chaineqn}
\partial\pi_*(\widetilde{\mathcal{OC}}_0(\widetilde{b}^{\otimes p}\otimes\widetilde{\mathfrak{m}}_k(\widetilde{b}^{\otimes k})\otimes\widetilde{b}^{\otimes q}\otimes\widetilde{b}) = \pi_*(\partial\widetilde{\mathcal{OC}}(\widetilde{b}^{\otimes p}\otimes\widetilde{\mathfrak{m}}_k(\widetilde{b}^{\otimes k})\otimes\widetilde{b}^{\otimes q}\otimes\widetilde{b}))
\end{align}
where the expression on the left hand side can be written as
\begin{align}
&\partial\pi_*(\widetilde{\mathcal{OC}}_0(\widetilde{b}^{\otimes p}\otimes\widetilde{\mathfrak{m}}_k(\widetilde{b}^{\otimes k})\otimes\widetilde{b}^{\otimes q}\otimes\widetilde{b}) \nonumber \\
&= \mathcal{OC}_0^{1}(b_{1}^{\otimes p}\otimes\mathfrak{m}_k^{1}(b_{1}^{\otimes k})\otimes b_{1}^{\otimes q}\otimes b_{1}) - \mathcal{OC}_0^{-1}(b_{-1}^{\otimes p}\otimes\mathfrak{m}_k^{-1}(b_{-1}^{\otimes k})\otimes b_{-1}^{\otimes q}\otimes b_{-1}) \label{diff}
\end{align}
by construction of the chain map $\pi_*$.
\begin{figure}
\begin{tikzpicture}

\node (N) at (0,0) {};
\node (E) at (0,4) {$e_{\pm 1}$};

\begin{scope}[very thick,decoration={
	markings,
	mark=at position 0.75 with {\arrow{>}}}]
\draw[postaction={decorate}, thick] (N) -- (E);
\end{scope}
\fill[white] (0,0) circle (2cm);

\draw (2,0) arc (0:180:2);
\draw (-2,0) arc (180:360:2 and 0.6);
\draw[dashed] (2,0) arc (0:180:2 and 0.6);

\coordinate (A) at (270: 2 and 0.6);
\node[anchor = north] at (270: 2 and 0.6) {$\widetilde{\mathfrak{m}}_k(\widetilde{b}^{\otimes k})$};
\draw[fill = white] (A) circle (2pt);

\coordinate (B) at (225: 2 and 0.6);
\node[anchor = north] at (225: 2 and 0.6) {$\widetilde{b}$};
\draw[fill = black] (B) circle (2pt);

\coordinate (C) at (315:2 and 0.6);
\node[anchor = north] at (315: 2 and 0.6) {$\widetilde{b}$};
\draw[fill = black] (C) circle (2pt);

\coordinate (D) at (90:2 and 0.6);
\draw[fill = white] (D) circle (2pt);
\node[anchor = south] at (90:2 and 0.6) {$\widetilde{b}$};

\coordinate (E) at (45:2 and 0.6);
\draw[fill = black] (E) circle (2pt);
\node[anchor = south] at (45:2 and 0.6) {$\widetilde{b}$};

\coordinate (F) at (135:2 and 0.6);
\draw[fill = black] (F) circle (2pt);
\node[anchor = south] at (135:2 and 0.6) {$\widetilde{b}$};
\end{tikzpicture}\caption{Elements of~\eqref{1dimtruemain}. The marked points corresponding to the input and output of $\widetilde{\phi}$ are the white dots.}
\end{figure}
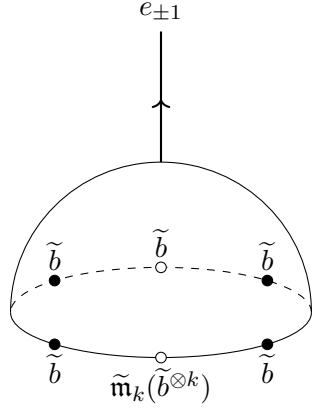

On the other hand, consider the one-dimensional moduli spaces of trees with underlying domain of open-closed type
\begin{align}
\overline{\mathcal{P}}_{0,p+q+1,\ell}(\widetilde{b}^{\otimes p}\otimes\widetilde{\mathfrak{m}}_k(\widetilde{b}^{\otimes k})\otimes\widetilde{b}^{\otimes q}\otimes\widetilde{b};e_{\pm 1}) . \label{1dimtruemain}
\end{align}

\begin{figure}
\begin{tikzpicture}

\node (N) at (0,0) {};
\node (E) at (0,4) {};
\node (S) at (4,4) {$e_{\pm 1}$};

\begin{scope}[very thick,decoration={
	markings,
	mark=at position 0.75 with {\arrow{>}}}]
\draw[postaction={decorate}, thick] (N) -- (E);
\end{scope}

\fill[black] (0,4) circle (3pt);

\begin{scope}[very thick,decoration={
	markings,
	mark=at position 0.5 with {\arrow{>}}}]
\draw[postaction={decorate}, thick] (E) -- (S);
\end{scope}

\fill[white] (0,0) circle (2cm);

\draw (2,0) arc (0:180:2);
\draw (-2,0) arc (180:360:2 and 0.6);
\draw[dashed] (2,0) arc (0:180:2 and 0.6);

\coordinate (A) at (270: 2 and 0.6);
\node[anchor = north] at (270: 2 and 0.6) {$\widetilde{\mathfrak{m}}_k(\widetilde{b}^{\otimes k})$};
\draw[fill = white] (A) circle (2pt);

\coordinate (B) at (225: 2 and 0.6);
\node[anchor = north] at (225: 2 and 0.6) {$\widetilde{b}$};
\draw[fill = black] (B) circle (2pt);

\coordinate (C) at (315:2 and 0.6);
\node[anchor = north] at (315: 2 and 0.6) {$\widetilde{b}$};
\draw[fill = black] (C) circle (2pt);

\coordinate (D) at (90:2 and 0.6);
\draw[fill = white] (D) circle (2pt);
\node[anchor = south] at (90:2 and 0.6) {$\widetilde{b}$};

\coordinate (E) at (45:2 and 0.6);
\draw[fill = black] (E) circle (2pt);
\node[anchor = south] at (45:2 and 0.6) {$\widetilde{b}$};

\coordinate (F) at (135:2 and 0.6);
\draw[fill = black] (F) circle (2pt);
\node[anchor = south] at (135:2 and 0.6) {$\widetilde{b}$};
\end{tikzpicture}\caption{Boundary stratum of~\eqref{1dimtruemain} involving breaking on the output edge.}~\label{intbreakstrata}
\end{figure}
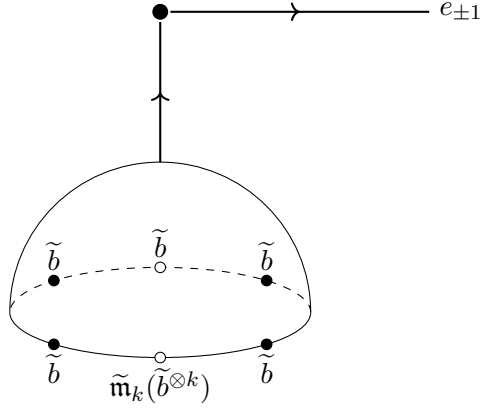
Examining the boundary strata of these spaces, we see that the boundary components where the output edge breaks into a broken negative gradient flow line in $M\times[-1,1]$ (cf.~\eqref{intbreakstrata}) contributes to the Morse differential of the open-closed map, i.e.
\[ \partial\widetilde{\mathcal{OC}}(\widetilde{b}^{\otimes p}\otimes\widetilde{\mathfrak{m}}_k(\widetilde{b}^{\otimes k})\otimes\widetilde{b}^{\otimes q}\otimes\widetilde{b}).\]
The other boundary strata involve breakings of gradient flow lines on $L\times[-1,1]$ or of pseudoholomorphic disks into nodal disks with boundary on $L\times[-1,1]$. These are represented schematically in Figures~\eqref{ksplitstrata},~\eqref{pqsplitstrata},~\eqref{outputstrata},  and~\eqref{cancelingstrata}. There is a parallel description of the moduli spaces
\begin{align}
\overline{\mathcal{P}}_{0,p+q+1,\ell}(\widetilde{b}^{\otimes p}\otimes\widetilde{b}\otimes\widetilde{b}^{\otimes q}\otimes\widetilde{\mathfrak{m}}_k(\widetilde{b}^{\otimes k});e_{\pm 1}) . \label{1dimtruemainalt}
\end{align}

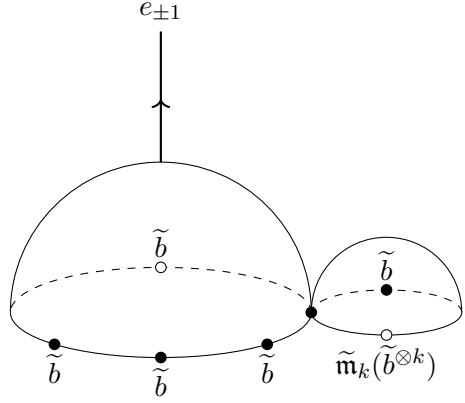
\begin{figure}
\begin{tikzpicture}
\begin{scope}
\node (N) at (0,0) {};
\node (E) at (0,4) {$e_{\pm 1}$};

\begin{scope}[very thick,decoration={
	markings,
	mark=at position 0.75 with {\arrow{>}}}]
\draw[postaction={decorate}, thick] (N) -- (E);
\end{scope}
\fill[white] (0,0) circle (2cm);

\draw (2,0) arc (0:180:2);
\draw (-2,0) arc (180:360:2 and 0.6);
\draw[dashed] (2,0) arc (0:180:2 and 0.6);

\begin{scope}[xshift = 3cm]
\draw (1,0) arc (0:180: 1 and 1);
\draw (-1,0) arc (180:360:1 and 0.3);
\draw[dashed] (1,0) arc (0:180:1 and 0.3);
\end{scope}

\coordinate (A) at (270: 2 and 0.6);
\node[anchor = north] at (270: 2 and 0.6) {$\widetilde{b}$};
\draw[fill = black] (A) circle (2pt);

\coordinate (B) at (225: 2 and 0.6);
\node[anchor = north] at (225: 2 and 0.6) {$\widetilde{b}$};
\draw[fill = black] (B) circle (2pt);

\coordinate (C) at (315:2 and 0.6);
\node[anchor = north] at (315: 2 and 0.6) {$\widetilde{b}$};
\draw[fill = black] (C) circle (2pt);

\coordinate (D) at (90:2 and 0.6);
\draw[fill = white] (D) circle (2pt);
\node[anchor = south] at (90:2 and 0.6) {$\widetilde{b}$};

\coordinate (G) at (0: 2 and 0.6);
\draw[fill = black] (G) circle (2pt);

\begin{scope}[xshift = 3cm]
\coordinate (E) at (90:1 and 0.3);
\draw[fill = black] (E) circle (2pt);
\node[anchor = south] at (90:1 and 0.3) {$\widetilde{b}$};

\coordinate (F) at (270:1 and 0.3);
\draw[fill = white] (F) circle (2pt);
\node[anchor = north] at (270:1 and 0.3) {$\widetilde{\mathfrak{m}}_k(\widetilde{b}^{\otimes k})$};
\end{scope}
\end{scope}
\end{tikzpicture}
\caption{Elements of the boundary stratum contributing to term~\eqref{ksplit}.}\label{ksplitstrata}
\end{figure}
 
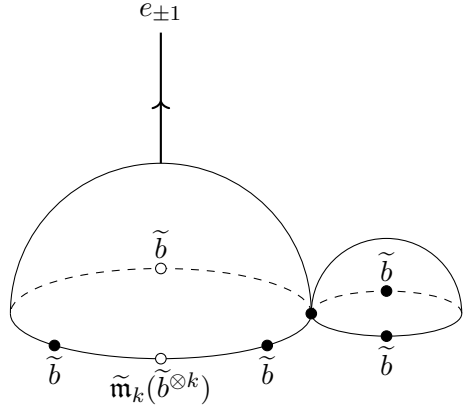
\begin{figure}
\begin{tikzpicture}
\begin{scope}
\node (N) at (0,0) {};
\node (E) at (0,4) {$e_{\pm 1}$};

\begin{scope}[very thick,decoration={
	markings,
	mark=at position 0.75 with {\arrow{>}}}]
\draw[postaction={decorate}, thick] (N) -- (E);
\end{scope}
\fill[white] (0,0) circle (2cm);

\draw (2,0) arc (0:180:2);
\draw (-2,0) arc (180:360:2 and 0.6);
\draw[dashed] (2,0) arc (0:180:2 and 0.6);

\begin{scope}[xshift = 3cm]
\draw (1,0) arc (0:180: 1 and 1);
\draw (-1,0) arc (180:360:1 and 0.3);
\draw[dashed] (1,0) arc (0:180:1 and 0.3);
\end{scope}

\coordinate (A) at (270: 2 and 0.6);
\node[anchor = north] at (270: 2 and 0.6) {$\widetilde{\mathfrak{m}}_k(\widetilde{b}^{\otimes k})$};
\draw[fill = white] (A) circle (2pt);

\coordinate (B) at (225: 2 and 0.6);
\node[anchor = north] at (225: 2 and 0.6) {$\widetilde{b}$};
\draw[fill = black] (B) circle (2pt);

\coordinate (C) at (315:2 and 0.6);
\node[anchor = north] at (315: 2 and 0.6) {$\widetilde{b}$};
\draw[fill = black] (C) circle (2pt);

\coordinate (D) at (90:2 and 0.6);
\draw[fill = white] (D) circle (2pt);
\node[anchor = south] at (90:2 and 0.6) {$\widetilde{b}$};

\coordinate (G) at (0: 2 and 0.6);
\draw[fill = black] (G) circle (2pt);

\begin{scope}[xshift = 3cm]
\coordinate (E) at (90:1 and 0.3);
\draw[fill = black] (E) circle (2pt);
\node[anchor = south] at (90:1 and 0.3) {$\widetilde{b}$};

\coordinate (F) at (270:1 and 0.3);
\draw[fill = black] (F) circle (2pt);
\node[anchor = north] at (270:1 and 0.3) {$\widetilde{b}$};
\end{scope}
\end{scope}
\end{tikzpicture}
\caption{Boundary strata contributing to terms~\eqref{psplit} and~\eqref{qsplit}.}\label{pqsplitstrata}
\end{figure}

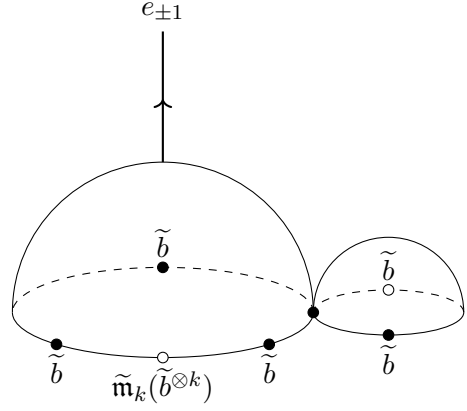
\begin{figure}
\begin{tikzpicture}
\begin{scope}
\node (N) at (0,0) {};
\node (E) at (0,4) {$e_{\pm 1}$};

\begin{scope}[very thick,decoration={
	markings,
	mark=at position 0.75 with {\arrow{>}}}]
\draw[postaction={decorate}, thick] (N) -- (E);
\end{scope}
\fill[white] (0,0) circle (2cm);

\draw (2,0) arc (0:180:2);
\draw (-2,0) arc (180:360:2 and 0.6);
\draw[dashed] (2,0) arc (0:180:2 and 0.6);

\begin{scope}[xshift = 3cm]
\draw (1,0) arc (0:180: 1 and 1);
\draw (-1,0) arc (180:360:1 and 0.3);
\draw[dashed] (1,0) arc (0:180:1 and 0.3);
\end{scope}

\coordinate (A) at (270: 2 and 0.6);
\node[anchor = north] at (270: 2 and 0.6) {$\widetilde{\mathfrak{m}}_k(\widetilde{b}^{\otimes k})$};
\draw[fill = white] (A) circle (2pt);

\coordinate (B) at (225: 2 and 0.6);
\node[anchor = north] at (225: 2 and 0.6) {$\widetilde{b}$};
\draw[fill = black] (B) circle (2pt);

\coordinate (C) at (315:2 and 0.6);
\node[anchor = north] at (315: 2 and 0.6) {$\widetilde{b}$};
\draw[fill = black] (C) circle (2pt);

\coordinate (D) at (90:2 and 0.6);
\draw[fill = black] (D) circle (2pt);
\node[anchor = south] at (90:2 and 0.6) {$\widetilde{b}$};

\coordinate (G) at (0: 2 and 0.6);
\draw[fill = black] (G) circle (2pt);

\begin{scope}[xshift = 3cm]
\coordinate (E) at (90:1 and 0.3);
\draw[fill = white] (E) circle (2pt);
\node[anchor = south] at (90:1 and 0.3) {$\widetilde{b}$};

\coordinate (F) at (270:1 and 0.3);
\draw[fill = black] (F) circle (2pt);
\node[anchor = north] at (270:1 and 0.3) {$\widetilde{b}$};
\end{scope}
\end{scope}
\end{tikzpicture}
\caption{Boundary stratum contributing to the term~\eqref{output}.}\label{outputstrata}
\end{figure}

\begin{figure}
\begin{tikzpicture}
\begin{scope}
\node (N) at (0,0) {};
\node (E) at (0,4) {$e_{\pm 1}$};

\begin{scope}[very thick,decoration={
	markings,
	mark=at position 0.75 with {\arrow{>}}}]
\draw[postaction={decorate}, thick] (N) -- (E);
\end{scope}
\fill[white] (0,0) circle (2cm);

\draw (2,0) arc (0:180:2);
\draw (-2,0) arc (180:360:2 and 0.6);
\draw[dashed] (2,0) arc (0:180:2 and 0.6);

\begin{scope}[xshift = 3cm]
\draw (1,0) arc (0:180: 1 and 1);
\draw (-1,0) arc (180:360:1 and 0.3);
\draw[dashed] (1,0) arc (0:180:1 and 0.3);
\end{scope}

\coordinate (A) at (270: 2 and 0.6);
\node[anchor = north] at (270: 2 and 0.6) {$\widetilde{b}$};
\draw[fill = black] (A) circle (2pt);

\coordinate (B) at (225: 2 and 0.6);
\node[anchor = north] at (225: 2 and 0.6) {$\widetilde{b}$};
\draw[fill = black] (B) circle (2pt);

\coordinate (C) at (315:2 and 0.6);
\node[anchor = north] at (315: 2 and 0.6) {$\widetilde{b}$};
\draw[fill = black] (C) circle (2pt);

\coordinate (D) at (90:2 and 0.6);
\draw[fill = black] (D) circle (2pt);
\node[anchor = south] at (90:2 and 0.6) {$\widetilde{b}$};

\coordinate (G) at (0: 2 and 0.6);
\draw[fill = black] (G) circle (2pt);

\begin{scope}[xshift = 3cm]
\coordinate (E) at (90:1 and 0.3);
\draw[fill = white] (E) circle (2pt);
\node[anchor = south] at (90:1 and 0.3) {$\widetilde{b}$};

\coordinate (F) at (270:1 and 0.3);
\draw[fill = white] (F) circle (2pt);
\node[anchor = north] at (270:1 and 0.3) {$\widetilde{\mathfrak{m}}_k(\widetilde{b}^{\otimes k})$};
\end{scope}
\end{scope}
\end{tikzpicture}
\caption{Canceling boundary strata.}\label{cancelingstrata}
\end{figure}
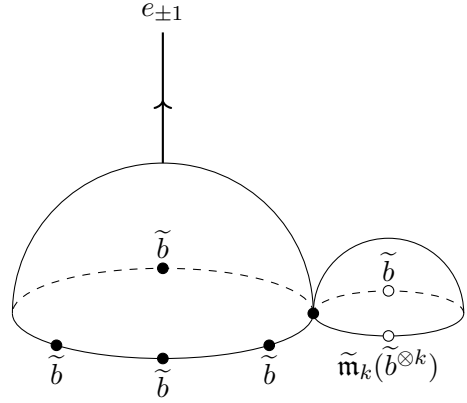

By~\eqref{diff} and the definition of the $\infty$-inner product, it follows that
\begin{align}
&\Phi'_1(b_1)-\Phi'_{-1}(b_{-}) \\
&=\sum_{N=0}^{\infty}\sum_{\substack{p+q+k  = N \\ k_1+k_2 = k+1}}\frac{1}{N+1}\sum_{r+s = k_1-1}\widetilde{\phi}_{p,q}(\widetilde{b}^{\otimes p}\otimes\underline{\widetilde{\mathfrak{m}}_{k_1}(\widetilde{b}^{\otimes r}\otimes\widetilde{\mathfrak{m}}_{k_2}(\widetilde{b}^{\otimes k_2})\otimes\widetilde{b}^{\otimes s})}\otimes\widetilde{b}^{\otimes q})(\widetilde{b}) \label{ksplit}\\
&+\sum_{N=0}^{\infty}\sum_{\substack{p+q+k  = N \\ k_1+k_2 = k+1}}\frac{1}{N+1}\sum_{r+s = p-1}\widetilde{\phi}_{p,q}(\widetilde{b}^{\otimes r}\otimes\widetilde{\mathfrak{m}}_{k_2}(\widetilde{b}^{\otimes k_2})\otimes\widetilde{b}^{\otimes s}\otimes\underline{\mathfrak{m}_{k_1}(\widetilde{b}^{\otimes k_1})}\otimes\widetilde{b}^{\otimes q})(\widetilde{b}) \label{psplit}\\
&+\sum_{N=0}^{\infty}\sum_{\substack{p+q+k  = N \\ k_1+k_2 = k+1}}\frac{1}{N+1}\sum_{r+s = q-1}\widetilde{\phi}_{p,q}(\widetilde{b}^{\otimes p}\otimes\underline{\widetilde{\mathfrak{m}}_{k_1}(\widetilde{b}^{\otimes k_1})}\otimes\widetilde{b}^{\otimes r}\otimes\widetilde{\mathfrak{m}}_{k_2}(\widetilde{b}^{\otimes k_2})\otimes\widetilde{b}^{\otimes s})(\widetilde{b}) \label{qsplit}\\
&+\sum_{N=0}^{\infty}\sum_{\substack{p+q+k  = N \\ k_1+k_2 = k+1}}\frac{k_2}{N+1}\widetilde{\phi}_{p,q}(\widetilde{b}^{\otimes p}\otimes\underline{\widetilde{\mathfrak{m}}_{k_1}(\widetilde{b}^{\otimes k_1})}\otimes\widetilde{b}^{\otimes q})(\widetilde{\mathfrak{m}}_{k_2}(\widetilde{b}^{\otimes k_2})) \label{output}
\end{align}
Because the value of $\widetilde{\phi}$ on inputs of the form under consideration are expressed as the difference
\[ \widetilde{\mathcal{OC}}_0(\widetilde{b}^{\otimes p}\otimes\widetilde{\mathfrak{m}}_k(\widetilde{b}^{\otimes k})\otimes\widetilde{b}^{\otimes q}\otimes\widetilde{b})-\widetilde{\mathcal{OC}}(\widetilde{b}^{\otimes p}\otimes\widetilde{b}\otimes\widetilde{b}^{\otimes q}\otimes\widetilde{\mathfrak{m}}_k(\widetilde{b}^{\otimes k}))\]
it follows that boundary strata of the sort depicted in Figure~\ref{cancelingstrata} appear in both~\eqref{1dimtruemain} and~\eqref{1dimtruemainalt}. Since $\widetilde{\phi}_{p,q}$ is defined by taking a difference corresponding to these two moduli spaces, the contributions of Figure~\eqref{cancelingstrata} cancel, and thus they do not appear in the sum above.

We can rewrite the sum of~\eqref{ksplit},~\eqref{psplit}, and~\eqref{qsplit} as
\begin{align}
&\sum_{N=0}^{\infty}\sum_{\substack{p+q+k  = N \\ k_1+k_2 = k+1}}\frac{N+1-k_2}{N+1}\widetilde{\phi}_{p,q}(\widetilde{b}^{\otimes p}\otimes\underline{\widetilde{\mathfrak{m}}_{k_1}(\widetilde{b}^{\otimes k_1})}\otimes\widetilde{b}^{\otimes q})(\widetilde{\mathfrak{m}}_{k_2}(\widetilde{b}^{\otimes k_2})) \label{clsum}
\end{align}
by Lemma~\ref{weakcyclic}. The sum of~\eqref{output} and~\eqref{clsum} can be rewritten, using the Maurer--Cartan equation, as
\begin{align}
\sum_{p,q\geq0}\widetilde{\phi}_{p,q}(\widetilde{b}^{\otimes p}\otimes\underline{(\widetilde{\mathfrak{m}}_0-c\cdot 1)}\otimes\widetilde{b}^{\otimes q})(\widetilde{\mathfrak{m}}_0-c\cdot 1)\, . \label{errortermv0}
\end{align}
Since $\Bbbk$ is of characteristic $0$, it follows from Lemma~\ref{KSanalogue} that the terms of~\eqref{errortermv0} for which $p>0$ or $q>0$ all vanish. Thus we are left with
\begin{align}
\widetilde{\phi}_{0,0}(\underline{\widetilde{\mathfrak{m}}_0})(\widetilde{\mathfrak{m}}_0) = \widetilde{\phi}_0(1,\widetilde{\mathfrak{m}}_2(\widetilde{\mathfrak{m}}_0,\widetilde{\mathfrak{m}}_0)) \, . \label{errortermv1}
\end{align}
by Lemma~\ref{homvstrace} and the linearity of $\widetilde{\phi}_{0,0}$. The right hand side of~\eqref{errortermv1} can itself be rewritten as
\begin{align}
\mathcal{OC}_0(\widetilde{\mathfrak{m}}_2(\widetilde{\mathfrak{m}}_0,\widetilde{\mathfrak{m}}_0)) \label{errortermv2}
\end{align}
by the construction of the negative cyclic cocycle.

We can also analyze the inhomogeneous terms similarly. One type of boundary component that can occur in the one-dimensional moduli spaces~\eqref{pearlytreescylm-1} consists of a broken configuration consisting of two pearly trees in~\eqref{cylpearls}, both of which have one output and no inputs, meaning that they would both contribute to $\mathfrak{m}_0$. Notice, however, that such broken configurations correspond exactly to those which contribute to~\eqref{errortermv2}, because $\widetilde{\mathfrak{m}}_2(\widetilde{\mathfrak{m}}_0,\widetilde{\mathfrak{m}}_0)$ is of top degree, so that only constant disks can contribute to the product.

Examining the remaining boundary strata of~\eqref{pearlytreescylm-1} shows that
\[ \mathfrak{m}_{-1}^{1}-\mathfrak{m}_{-1}^0+\mathcal{OC}_0(\widetilde{\mathfrak{m}}_2(\widetilde{\mathfrak{m}}_0,\widetilde{\mathfrak{m}}_0))+\widetilde{GW} = 0. \]
The third term in this sum coincides with~\eqref{errortermv2}, and the remaining terms give the leading terms and wall-crossing term.
\end{proof}
\begin{remark}
Our appeal to Lemma~\ref{KSanalogue} is the only place, other than in the definition of the infinity cyclic potential itself, where we use the fact that our ground field has characteristic $0$.
\end{remark}
\section{Comparison with Solomon and Tukachinsky's invariants}\label{comparison}
Showing that the open Gromov--Witten potential of Definition~\ref{infinityogwpotential} agrees with the open Gromov--Witten potential of~\cite{ST21} would most likely require the construction of a very well-behaved quasi-isomorphism between the pearly $A_{\infty}$-algebra for $L\subset M$ and the de Rham version of the $A_{\infty}$-algebra for $L$. Alternatively, one might hope to compare our invariants with those of~\cite{Che22a}. Instead of pursuing this, we will sketch the analogue of our construction under the technical assumptions of~\cite{ST21}, illustrating in the process how our constructions simplify in the presence of a strictly cyclic pairing.

Recall that~\cite{ST21} assumes that
\begin{itemize}
\item[•] the moduli spaces $\mathcal{M}_{k+1,\ell}(\beta)$ are compact orbifolds with corners for all $k\geq-1$ and $\ell\geq0$ and the boundary evaluation maps $\evb_0$ on these spaces at the zeroth boundary marked points are submersions.
\end{itemize}
Under this assumption, the closed-open operations $\mathfrak{q}_{k,\ell}^{\beta}$ are defined by pulling back differential forms $\alpha_1,\ldots,\alpha_k\in\Omega^*(L;R)$ and $\gamma_1,\ldots,\gamma_{\ell}\in\Omega^*(M;Q)$ to $\mathcal{M}_{k+1,\ell}(\beta)$ under the corresponding evaluation maps, taking the wedge product of these forms, and pushing forward by $\evb_0$. Here the pushforward of differential forms is given by integration along the fiber, which is where submersivity is required. These yield bulk-deformed $A_{\infty}$-operations on $\Omega^*(L;R)$. 

Let $\gamma\in\Omega^*(L;R)$ be a bulk parameter, and let $\mathcal{A}$ denote $\Omega^*(L;R)$ equipped with the resulting $A_{\infty}$-operations. In~\cite{Hug24}, the cyclic open-closed map on the de Rham complex was constructed under these regularity assumptions. There, the open-closed map $\mathcal{OC}_0$ is characterized by the property that
\begin{equation} \langle\eta,\mathcal{OC}_0(\alpha)\rangle_M = (-1)^{|\alpha_0|(\sum_{i\geq 1}|\alpha_i|'+1)}\langle\mathfrak{q}^{\gamma}_{k,1}(\alpha_1\otimes\cdots\otimes\alpha_k;\eta),\alpha_0\rangle_L \label{hugoc}
\end{equation}
for any reduced Hochschild cohain $\alpha = \underline{\alpha_0}\otimes\alpha_1\otimes\cdots\otimes\alpha_k\in\mathcal{A}\otimes(\mathcal{A}[1])^{\otimes k}$ and any differential form $\eta\in\Omega^*(M;Q)$. Since the integration pairing on $L$ is strictly cyclic, we can extend the open-closed map $u$-linearly to obtain a cyclic open-closed map $\mathcal{OC}$. This induces a strictly cyclic $\infty$-inner product $\psi$ on $\mathcal{A}$, whose values are determined only by the open-closed map.

The potential of Definition~\ref{infinityogwpotential} in this case reduces to
\[\Psi(b) = \mathfrak{m}_{-1}^{\gamma}+\sum_{k=0}^{\infty}\frac{1}{k+1}\psi_{0,0}(\underline{\mathfrak{m}^{\gamma}_k(b^{\otimes k})})(b). \]
Since we have obtained $\psi$ from a negative cyclic cocycle, it follows that
\[ \psi_{0,0}((\underline{\mathfrak{m}^{\gamma}_k(b^{\otimes k})})(b) = \psi_0(1,\mathfrak{m}_2^{\gamma}(\mathfrak{m}^{\gamma}_k(b^{\otimes k}),b)) \]
where $\psi_0$ refers to the part of the negative cyclic cocycle residing in the zeroth column of the $(b^*,B^*)$-bicomplex~\eqref{b*B*complex}. Since the negative cyclic cocyle is obtained from the cyclic open-closed map under the isomorphism of Lemma~\ref{hciso}, it follows that the open Gromov-Witten potential can be rewritten as
\[\Psi(b) = \mathfrak{m}_{-1}^{\gamma}+\sum_{k=0}^{\infty}\frac{1}{k+1}\mathcal{OC}_0(\mathfrak{m}_2^{\gamma}(\mathfrak{m}_k^{\gamma}(b^{\otimes k})\otimes b)). \]
By the top-degree property~\cite[Proposition 3.12]{ST22} of Solomon and Tukachinsky's $A_{\infty}$-algebra, we have that
\[ \mathfrak{m}_2^{\gamma}(\mathfrak{m}_k^{\gamma}(b^{\otimes k})\otimes b) = \mathfrak{m}_k^{\gamma}(b^{\otimes k})\wedge b.\]
Using~\eqref{hugoc}, we compute
\begin{align*}
\langle 1,\mathcal{OC}_0(\mathfrak{m}_k^{\gamma}(b^{\otimes k})\wedge b)\rangle_M = \langle\mathfrak{q}_{0,1}^{\gamma}(1),\mathfrak{m}_k^{\gamma}(b^{\otimes k})\wedge b\rangle_L = \langle\mathfrak{m}_k^{\gamma}(b^{\otimes k}),b\rangle_L.
\end{align*}
To summarize, we have proven the following.
\begin{theorem}
For any $L\subset M$ subject to the assumptions of~\cite{ST21}, the $\infty$-OGW potential defined over the de Rham complex recovers the OGW potential of~\cite{ST21} up to an overall sign.
\end{theorem}
It is also possible to give an independent proof of the wall-crossing formula over the de Rham complex using pseudo-isotopies of $A_{\infty}$-algebras defined in~\cite{F10} or~\cite{F11}. By~\cite[Lemma 21.31]{FOOO20}, these arise from $A_{\infty}$-structures on the de Rham complex of $L\times I$ as constructed in~\cite{ST22}.

\appendix
\section{Regularity hypotheses}\label{regappendix} We have made several regularity assumptions throughout the main body of this paper, all of which are summarized below.
\begin{itemize}
\item[•] Assumption~\ref{uniqueminmax}: all Morse functions used to define Morse (co)chain complexes of $L$ and $M$ have a unique local minimum and a unique local maximum.

\item[•] Assumption~\ref{regularity}: there is a $J\in\mathcal{J}(M)$ such that the moduli spaces of pseudoholomorphic pearly trees~\eqref{pearls} of virtual dimension at most $1$ are transversely cut out orbifolds of the expected dimension.

\item[•] Assumption~\ref{cylregularity}: for any two $J_{\pm 1}$ satisfying assumption~\ref{regularity}, there is a path $\underline{J} = \lbrace J_t\rbrace_{t\in[-1,1]}$ such that the moduli spaces~\eqref{cylpearls} of virtual dimension at most $1$ are transversely cut out orbifolds of the expected dimension. Note that the definition of these moduli spaces requires that we have constructed Morse--Smale pairs on $L\times[-1,1]$ and $M\times[-1,1]$, as detailed in Section~\ref{cylinderobjects}.

\item[•] Assumption~\ref{regularityoc}: the moduli spaces of open-closed pearly trees~\eqref{ocpearls} and~\eqref{s1pearls} of virtual dimension at most $1$ are transversely cut out orbifolds of the expected dimension. Here the moduli spaces are defined using the same almost complex structure of Assumption~\ref{regularity} appearing in the definition of the $A_{\infty}$-operations.

\item[•] Assumption~\ref{occylinderregularity}: the moduli spaces of open-closed pearly trees on the cylinder~\eqref{ocpearlscylinder} and~\eqref{s1pearlscylinder} of virtual dimension at most $1$ are transversely cut out orbifolds of the expected dimension. Here the moduli spaces are defined using the same path of almost complex structures of Assumption~\ref{cylregularity}.

\item[•] Assumption~\ref{m-1regularity}: the moduli spaces~\eqref{pearlytreesm-1} of $J$-holomorphic pearly trees with no inputs in $L$ of virtual dimension $0$ are transversely cut out $0$-dimensional manifolds.

\item[•] Assumption~\ref{m-1cylregularity}: the moduli spaces~\eqref{pearlytreescylm-1} of $\underline{J}$-holomorphic pearly trees with no inputs in $L\times[-1,1]$ of virtual dimension at most $1$ are transversely cut out orbifolds of the expected dimension.
 
\item[•] Assumption~\ref{cylsphereregularity}: the moduli spaces~\eqref{GWmod} of $\underline{J}$-holomorphic pearly trees in $M\times[-1,1]$ with only sphere components of virtual dimension at most $1$ are transversely cut out orbifolds of the expected dimension.
\end{itemize}
We remark that there exist $J\in\mathcal{J}(M)$ and paths $\underline{J} = \lbrace J_t\in\mathcal{J}(M)\rbrace_{t\in[-1,1]}$ simultaneously satisfying all of these assumptions if one of the following conditions on $L\subset M$ is satisfied.
\begin{itemize}
\item[(i)] $L\subset M$ satisfies the assumptions of~\cite{ST21} for some fixed $J$. For the time-dependent moduli spaces, there should exist a path $\underline{J}$ in $\mathcal{J}(M)$ which satisfies the assumptions of~\cite{ST21} at all times.
\item[(ii)] $L$ is a monotone Lagrangian in a monotone symplectic manifold and $J\in\mathcal{J}(M)$ and $\underline{J} = \lbrace J_t\in\mathcal{J}(M)\colon t\in[-1,1]\rbrace$ are generic.
\end{itemize}
In the case of (i), the assumptions of~\cite{ST21} imply that all of the moduli spaces $\mathcal{M}_{k+1,\ell}(\beta;J)$ is pseudo-holomorphic disks in $M$ with boundary on $L$ are already smooth orbifolds with corners of the expected dimension. Thus it is clear that one can choose Morse functions satisfying Assumption~\ref{uniqueminmax} for which Assumptions~\ref{regularity} and~\ref{m-1regularity} are satisfied. The spaces of disks~\ref{ocmodbeta} and~\ref{s1modbeta} can be identified with subdomains of moduli spaces already covered by the assumptions of~\cite{ST21}, giving us Assumption~\ref{regularityoc} immediately. The assumptions involving pearly trees in $M\times[-1,1]$ can also be checked similarly.

In the monotone case (ii), Assumption~\ref{regularity} can be checked using the techniques of~\cite{BC07}, with no modifications. Roughly, this works by decomposing any $J$-holomorphic disk on $L$ as a sum of simple disks in homology, and then using the constraint on virtual dimensions to argue that all pearly trees contributing to the $A_{\infty}$-operations are equipped with simple disks at all vertices. The verification of all other assumptions on pearly trees of disks can be carried out in the same way. In particular, the extra decorations on the domains of the open-closed moduli spaces introduce no additional complications. Assumption~\ref{cylsphereregularity} in the monotone setting follows from the discussion of the quantum product in~\cite{BC07}.
\bibliographystyle{abbrv}
\bibliography{myref}

\end{document}